\documentclass{amsart}
\usepackage{pdfsync}
\usepackage[linktocpage]{hyperref}
\usepackage{amssymb,paralist,xspace,graphicx,url,amscd,euscript,mathrsfs,stmaryrd,epic,eepic,color}
\usepackage[all]{xy}
\usepackage{amsthm,amsmath,amsfonts,amsopn,ifthen,amscd,latexsym,mathdots}
\usepackage{enumerate}
\usepackage{lipsum}
\usepackage{MnSymbol}

\SelectTips{cm}{}
\numberwithin{equation}{section}
1
\numberwithin{subsection}{section}

\allowdisplaybreaks[1]

\newtheorem*{namedtheorem}{\theoremname}
\newcommand{\theoremname}{testing}

\newtheorem{theorem}[subsection]{Theorem}
\newtheorem{proposition}[subsection]{Proposition}
\newtheorem{proposition-definition}[subsection]
{Proposition-Definition}
\newtheorem{corollary}[subsection]{Corollary}
\newtheorem{lemma}[subsection]{Lemma}
\theoremstyle{definition}
\newtheorem{definition}[subsection]{Definition}

\newtheorem{example}[subsection]{Example}
\newtheorem{remark}[subsection]{Remark}

\newtheorem*{que}{Question}

\theoremstyle{remark}

\newcommand\cA{\mathcal{A}}

\newcommand\cK{\mathcal{K}}

\newcommand\cM{\mathcal{M}}

\newcommand\cO{\mathcal{O}}

\newcommand\cS{\mathcal{S}}

\renewcommand\AA{\mathbb{A}}

\newcommand\CC{\mathbb{C}}

\newcommand\HH{\mathbb{H}}

\newcommand\LL{\mathbb{L}}

\newcommand\PP{\mathbb{P}}
\newcommand\QQ{\mathbb{Q}}
\newcommand\RR{\mathbb{R}}
\renewcommand\SS{\mathbb{S}}

\newcommand\ZZ{\mathbb{Z}}

\newcommand\fra{\mathfrak{a}}

\newcommand\frg{\mathfrak{g}}
\newcommand\frh{\mathfrak{h}}

\newcommand\frk{\mathfrak{k}}
\newcommand\frl{\mathfrak{l}}

\newcommand\frp{\mathfrak{p}}
\newcommand\frq{\mathfrak{q}}

\newcommand\frt{\mathfrak{t}}
\newcommand\fru{\mathfrak{u}}

\newcommand{\rmI}{{\rm I}}

\newcommand\rank{{\rm rank}}
\newcommand\Pic{{\rm Pic}}

\newcommand{\SO}{{\rm SO}}

\newcommand{\tK}{\widetilde{K}}
\newcommand{\bv}{{\bf x}}

\newcommand{\Mp}{{\rm Mp}}
\newcommand{\Sp}{{\rm Sp}}
\newcommand{\GL}{{\rm GL}}
\newcommand{\SL}{{\rm SL}}

\newcommand{\Hom}{{\rm Hom}}
\newcommand{\Gspin}{{\rm GSpin}}

\newcommand{\ind}{{\rm ind}}

\theoremstyle{plain}
\theoremstyle{definition}

\begin{document}

\title[The Noether-Lefschetz conjecture]{The Noether-Lefschetz conjecture and generalizations}
\author{Nicolas Bergeron} 
\thanks{N.B. is a member of the Institut Universitaire de France. J.M. was supported by NSF grant DMS -1206999}

\address{Institut de Math\'ematiques de Jussieu \\
Unit\'e Mixte de Recherche 7586 du CNRS \\
Universit\'e Pierre et Marie Curie \\
4, place Jussieu 75252 Paris Cedex 05, France \\}
\email{bergeron@math.jussieu.fr}
\urladdr{http://people.math.jussieu.fr/~bergeron}

\author{Zhiyuan Li}
\address{Department of Mathematics, Stanford University\\
building 380\\
Stanford, CA 94305\\
U.S.A.}
\email{zli2@stanford.edu}

\author{John Millson}
\address{Department of Mathematics\\
University of Maryland\\
College Park, Maryland
20742, USA }
\email{jjm@math.umd.edu}
\urladdr{http://www-users.math.umd.edu/~jjm/}

\author{Colette Moeglin}
\address{Institut de Math\'ematiques de Jussieu \\
Unit\'e Mixte de Recherche 7586 du CNRS \\
4, place Jussieu 75252 Paris Cedex 05, France \\}
\email{moeglin@math.jussieu.fr}
\urladdr{http://www.math.jussieu.fr/~moeglin}

\begin{abstract}
We prove the Noether-Lefschetz conjecture on the moduli space of quasi-polarized K3 surfaces. This is deduced as a particular case of a general theorem that states that low degree cohomology classes of arithmetic manifolds of orthogonal type are  dual to the classes of special cycles, i.e. sub-arithmetic manifolds of the same type. For compact manifolds this was proved in \cite{BMM11}, here we extend the results of \cite{BMM11} to non-compact manifolds. This allows us to apply our results to the moduli spaces of quasi-polarized K3 surfaces. 
\end{abstract}
\maketitle

\tableofcontents

\section{Introduction}

\subsection{The Noether-Lefschetz conjecture}
The  study of Picard groups of moduli problems was started by Mumford \cite{Mum65} in the 1960's.   For the  moduli space $\cM_g$ of genus $g$ curves,   Mumford and Harer  (cf.~ \cite{Mum67, Ha83}) showed that the Picard group $\Pic(\cM_g)$ of $\cM_g$  is isomorphic to its second cohomology group $H^2(\cM_g,\ZZ)$, which is a finitely generated abelian group of rank one  for $g\ge 3$. Moreover, the generator of $H^2(\cM_g,\QQ)$  is the first Chern class of the Hodge bundle on $\cM_g$. 

In higher dimensional moduli theory, a {\it quasi-polarized} K3 surface of genus $g$ is a two dimensional analogue of the genus $g$ smooth projective curve.   
Here,  a  quasi-polarized K3 surface of genus $g\geq 2$ is defined by a pair $(S,L)$  where $S$ is a K3 surface and $L$ is line bundle on $S$ with primitive Chern class $c_1(L)\in H^2(S,\ZZ)$ satisfying
$$ L \cdot L = \int_S c_1(L)^2 =2g-2 \quad \mbox{and} \quad L \cdot C  = \int_C c_1(L) \geq 0$$
for every curve $C \subset S$.  Let $\cK_g$ be the coarse moduli space of quasi-polarized K3 surfaces of genus $g$.  Unlike the case of $\Pic(\cM_g)$, O'Grady \cite{OG86} has shown that the rank of $\Pic(\cK_g)$ can be arbitrarily large. 
Besides the Hodge line bundle,  there are actually many other natural divisors on $\cK_g$  coming from Noether-Lefschetz theory developed by Griffiths and Harris in \cite{GH85} (see also \cite{MP13}).  More precisely, the Noether-Lefschetz locus  in $\cK_g$ parametrizes K3 surfaces in $\cK_{g}$ with Picard number greater than $2$; it is a countable union of divisors. Each of them parametrizes the K3 surfaces whose Picard lattice contains a special curve class; these divisors  are called {\it Noether-Lefschetz (NL) divisors} on $\cK_g$. Oguiso's theorem \cite[Main Theorem]{Og00} implies that any curve on $\cK_g$ will meet some NL-divisor on $\cK_g$ (see also \cite[Theorem 1.1]{BKPS98}).   
%
So it is natural to ask whether the Picard group $\Pic_\QQ(\cK_g)$ of $\cK_{g}$ with rational coefficients is spanned by NL-divisors.  This is conjectured to be true by Maulik and Pandharipande, see \cite[Conjecture 3]{MP13}. More generally, one can extend this question to higher NL-loci on $\cK_g$, which parametrize K3 surfaces in $\cK_g$ with higher Picard number, see \cite{Ku13}.  Call the irreducible components of higher NL-loci the NL-cycles on $\cK_g$. Each of them parametrizes K3 surfaces in $\cK_g$ whose Picard lattice contains a special primitive lattice. 

\begin{theorem} \label{NL-conj}
For all $g \geq 2$ and all $r\leq 4$, the cohomology group $H^{2r}(\cK_g , \QQ)$ is spanned by NL-cyles of codimension $r$. In particular (taking $r=1$), $\Pic_\QQ(\cK_g)\cong H^2(\cK_g,\QQ)$ and the Noether-Lefschetz conjecture holds on $\cK_{g}$ for all $g \geq 2$.
\end{theorem}

\begin{remark}
There is a purely geometric approach (cf.~\cite{GLT14}) for low genus case ($g\leq 12$), but it can not be applied for large genera. It remains interesting to give a geometric proof for this conjecture.
\end{remark}

Combined with works of Borcherds and Bruinier in \cite{Bo98} and \cite{Br02} (see also \cite{LT13}),  we get the following: 

\begin{corollary}
We have
\begin{multline}\label{rank}
\rank (\Pic(\cK_g))= \frac{31g+24}{24}-\frac{1}{4}\alpha_g-\frac{1}{6}\beta_g  \\-  \sum\limits_{k=0}^{g-1}\left\{\frac{k^2}{4g-4}\right\}-\sharp\left\{k~|~\frac{k^2}{4g-4}\in\ZZ, 0\leq k\leq g-1 \right\} 
\end{multline}
where 
$$\alpha_g= \begin{cases}
     0, & \text{if $g$ is even}, \\
      \left( \frac{2g-2}{2g-3}\right)& \text{otherwise},
\end{cases} \quad \beta_g= \begin{cases}
     \left( \frac{g-1}{4g-5}\right)-1, & \text{if $g\equiv 1\mod 3$}, \\
   \left( \frac{g-1}{4g-5}\right)+\left( \frac{g-1}{3}\right)   & \text{otherwise},
\end{cases}$$ and $\left( \frac{a}{b}\right)$ is the Jacobi symbol.
\end{corollary}

\subsection{From moduli theory to Shimura varieties of orthogonal type}

Theorem \ref{NL-conj} is deduced from a general theorem on arithmetic manifolds. Indeed: let $(S,L)$ be a  K3 surface in $\cK_g$,  then the middle cohomology $H^2(S,\ZZ)$ is an even unimodular lattice of signature $(3,19)$ under the intersection form $\left<,\right>$ which is isometric to the K3 lattice
$$L_{K3} = U^{\oplus 3} \oplus (-E_8)^{\oplus 2},$$
where $U$ is the hyperbolic lattice of rank two and $E_8$ is the positive definite lattice associated to the Lie group of the same name; see \cite{MP13}.  In fact a fundamental theorem of M. Friedman states that the homeomorphism type of a compact, orientable, simply connected (real) $4$-manifold $M$ is determined by the isometry type of $(H^2 (M, \ZZ) , \left< , \right>)$. One may then {\it define} a K3 surface to be a simply connected algebraic surface $S$ whose middle cohomology lattice $(H^2(S,\ZZ),  \left< , \right>)$ is isometric to the K3 lattice $L_{K3}$. Any nonsingular quartic in $\mathbb{P}^3$ is a K3 surface. 

A {\it marking} on a K3 surface $S$ is a choice of an isometry $u : H^2 (S , \ZZ ) \to L_{K3}$. If $(S,L)$ is a quasi-polarized K3 surface, the first Chern class $c_1 (L)$ is a primitive  vector in $H^2 (S, \ZZ)$.  Define the primitive sublattice  $H^2(S,\ZZ)_{\rm prim}$ of 
$H^2(S,\ZZ)$ by 
$$H^2(S,\ZZ)_{\rm prim} = \{ \eta \in  H^2(S,\ZZ): \eta \wedge c_1(L) = 0\}.$$
Then we have an orthogonal (for the intersection form) splitting 
\begin{equation} \label{Lefschetzdecom}
 H^2(S,\ZZ) = \ZZ c_1(L) + H^2(S,\ZZ)_{\rm prim}.
\end{equation}
There is a Hodge structure on $H^2(S,\ZZ)_{\rm prim}$ given by the Hodge decomposition induced by the Hodge structure on $H^2(S,\ZZ)$:
$$H^2 (S,\ZZ)_{\rm prim} \otimes_{\ZZ} \CC = H^{2,0} (S , \CC) \oplus H^{1,1} (S , \CC)_{\rm prim} \oplus H^{0,2} (S , \CC)$$ 
with Hodge number $(1,19,1)$.  

We will now describe  the moduli space of such polarized Hodge structures.  Fix a {\it primitive} element $v \in L_{K3}$ such that $v^2$ is positive (and therefore equal to $2(g-1)$ for some $g \geq 2$). Write 
$$L_{K3} = \left< v \right> \oplus^{\perp} \Lambda.$$
The lattice $\Lambda$ is then isometric to the even lattice
$$\ZZ w \oplus U^{\oplus 2}\oplus (-E_8)^{\oplus 2},$$
where $\langle w , w\rangle = 2-2g$. A {\it marked $v$-quasi-polarized} K3 surface is a collection $(S,L, u )$ where $(S,L)$ is a quasi-polarized K3 surface, $u$ is a marking and $u(c_1 (L)) = v$. Note that this forces $(S,L)$ to be of genus $g$. 
The {\it period point} of $(S,L, u )$ is $u_{\CC} (H^{2,0} (S))$, where $u_\CC : H^2 (S , \ZZ) \otimes_{\ZZ} \CC \to (L_{K3})_\CC$ is the complex linear extension of $u$. It is a complex line $\CC \cdot \omega \in (L_{K3})_{\CC}$ satisfying \begin{equation*}
\langle \omega , \omega \rangle =0 \quad \mbox{and} \quad \langle \omega , \overline{\omega} \rangle >0.
\end{equation*}
It is moreover orthogonal to $v = u(c_1 (L)) \in L_{K3}$. Let $V = \Lambda \otimes_\ZZ \QQ$ be the linear space orthogonal to $v$. We conclude that the period point belongs to  
\begin{equation*}
\begin{split}
D=D(V) & = \{ \omega \in V \otimes_\QQ \CC \; | \; \langle \omega , \omega \rangle =0, \ \langle \omega , \overline{\omega} \rangle >0 \} / \CC^{\times} \\
 & \cong \{ \mbox{oriented positive $2$-planes in } V_\RR = V\otimes_\QQ \RR \} \\
 & \cong \mathrm{SO} (V_\RR ) / K_\RR , 
\end{split}
\end{equation*}
where $K_\RR \cong \SO (2) \times \SO(19)$ is the stabilizer of an oriented positive $2$-plane in $\SO(V_\RR) \cong \SO (2,19)$. Here, by { \it positive} $2$-plane 
$P$ we mean a two dimensional subspace  $P \subset V_{\mathbb{R}}$ such that the restriction of the form $\langle \  , \  \rangle$ to $P$ is positive definite.
 
Now the global Torelli theorem for K3 surfaces (cf.~\cite{PS71, Fri84}) says that the period map is onto and that if $(S,L)$ and $(S',L')$ are two quasi-polarized K3-surfaces and if there exists an isometry of lattices $\psi : H^2 (S' , \ZZ) \to H^2 (S , \ZZ )$ such that 
$$\psi (c_1 (L' )) = \psi (c_1 (L )) \quad \mbox{and} \quad \psi_{\CC} (H^{2,0} (S')) = H^{2,0} (S),$$
then there exists a unique isomorphism of algebraic varieties $f : S \to S'$ such that $f^* = \psi$. Forgetting the marking, we conclude that the period map identifies the complex points of the moduli space $\cK_g$ with the quotient  $\Gamma\backslash D$ where 
$$\Gamma = \{ \gamma \in \SO (\Lambda ) \; | \; \gamma ~\hbox{acts trivially on  $\Lambda^{\vee}/ \Lambda$}  \},$$
is the natural monodromy group acting properly discontinuously on $D$. 
 
Note that the quotient $Y_\Gamma = \Gamma \backslash D$ is a connected component of a Shimura variety associated to the group $\SO(2,19)$.

We now interpret NL-cycles on $\cK_g$ as special cycles on $Y_\Gamma $.   Fix a vector $x$ in $\Lambda$.   Then the set of marked $v$-quasi-polarized K3 surfaces 
$(S,L , u ) $ for which $x$ is the projection in $\Lambda$ of an additional element in
$$\mathrm{Pic} (S)\cong H^{1,1}(S,\CC)\cap H^{2}(S,\ZZ)$$
corresponds to the subset of $(S,L , u )\in \cK_g$ for which the  period point $u_{\CC} (H^{2,0} (S))$ belongs to 
\begin{equation*}
\begin{split}
D_x  & = \{ \underline{\omega} \in D \; | \; \langle x , \omega \rangle = 0 \} = D(V \cap x^{\perp}) \\
& = \{ \mbox{oriented positive $2$-planes in } V_\RR \mbox{ that are orthogonal to } x \}.
\end{split}
\end{equation*} 
By the Hodge index theorem we have $x^2 \leq 0$. In case $x^2=0$ the subspace $D_x$ of $D$ defined above is empty. In this case, the analogue of the union of all $D_x$ with $x^2=0$ will be defined below. The image in $Y_\Gamma$ of the NL-locus that parametrizes K3 surfaces with Picard number $\geq 2$ is therefore 
the union of the divisors obtained by projecting the $D_x$'s. Maulik and Pandharipande define refined divisors by specifying a Picard class: fixing two integers $h$ and $d$ such that 
$$\Delta (h,d) := - \det \left( \begin{array}{cc}
2g-2 & d \\
d & 2h-2 
\end{array} \right) = d^2 - 4(g-1)(h-1) >0,$$
Maulik and Pandharipande more precisely define the NL-divisor $D_{h,d}$ to have support
on the locus of quasi-polarized K3 surfaces for which there exists a class $\beta \in \mathrm{Pic} (S)$ corresponding to a divisor $D$ satisfying 
$$D \cdot D = \int_S \beta^2 = 2h-2 \quad \mbox{and} \quad D \cdot L = \int_S \beta \wedge c_1(L) =d.$$
Let 
$$n= - \frac{\Delta (h,d)}{4(g-1)}, \quad \gamma = d \left( \frac{1}{2g-2} w \right) \in \Lambda^{\vee} / \Lambda,$$
and
$$\Omega_n = \{ x \in \Lambda^{\vee} \; | \; \frac12 \langle x , x \rangle = n\}.$$
The group $\Gamma$ acts on $\Omega_n$ with finitely many orbits and \cite[Lemma 3, p. 30]{MP13} implies that (the image of) $D_{h,d}$ in $Y_\Gamma$ is 
\begin{equation} \label{NLdiv}
D_{h,d} = \sum_{\substack{x \in \Omega_n \mod \Gamma \\ x \equiv \gamma \mod \Lambda}}  \Gamma_x \backslash D_x,
\end{equation}
where $\Gamma_x$ is the stabilizer of $D_x$ in $\Gamma$. Note that $w/(2g-2)$ is a generator of $\Lambda^{\vee} / \Lambda$ so that any $\gamma$ can be obtained as $d$ varies. In \eqref{NLdiv} the rational $n$ is positive. 
However the theory of Kudla-Millson suggests that when $n=0$ the class of $D_{h,d}$ should be replaced by the Euler class, which is also the class of the Hodge line bundle. We shall show in \S 8, Corollary \ref{Euler-special}, that this class belongs to the span of the special cycles.  

The divisors \eqref{NLdiv} are particular cases of the {\it special cycles} that we define in the general context of arithmetic manifolds associated to quadratic forms in \S \ref{S8}.
When the arithmetic manifold is $Y_\Gamma$ as above, codimension 1 special cycles span the same subspace of the cohomology as the classes of the NL-divisors \eqref{NLdiv}.  
By ``NL-cycles of codimension $r$'' we refer to codimension $r$ special cycles. See Kudla \cite[Proposition 3.2]{Ku13} for relations with the Noether-Lefschetz theory. Our main result will be  
more generally concerned with special cycles in general non-compact Shimura varieties associated to orthogonal groups.  

\begin{remark} General (non-compact) Shimura varieties associated to $\Lambda$ correspond to the moduli spaces of primitively quasi-polarized K3 surfaces with level structures; see \cite[$\S$2]{Ri10}. Recently, these moduli spaces play a more and more important role in the study of K3 surfaces (cf.~\cite{Pe13,Ma12}). The Hodge-type result above can be naturally extended to these moduli spaces.
\end{remark}

\subsection{Shimura varieties of orthogonal type}
Let $Y$ be a connected smooth Shimura  variety of orthogonal type, that is, a congruence locally Hermitian symmetric variety associated to the orthogonal group $\SO(p,2)$.   One attractive feature of these Shimura varieties is that they have many  algebraic cycles coming from sub-Shimura varieties of the same type in all codimensions. These are the so called special cycles on $Y$ and  play the central role in Kudla's program, see \S \ref{S8} for a precise definition.  A natural question arising from geometry and also arithmetic group cohomology theory is:
\begin{que} \label{Q}
{\it Do the classes of special cycles of codimension $r$ exhaust all the cohomology classes in $H^{2r}(Y,\QQ) \cap H^{r,r}(Y)$ for sufficiently small $r$}?
\end{que}
Note that for $2r<\frac{p}{2}$, the cohomology group $H^{2r}(Y,\QQ)$ has a pure Hodge structure of weight $2r$ with only Hodge classes (see Example \ref{zucker-conj}), i.e. $H^{2r}(Y,\CC)=H^{r,r}(Y)$, so that we can simply replace $H^{2r}(Y,\QQ) \cap H^{r,r}(Y)$ by $H^{2r}(Y,\QQ)$ in the question above. When $Y$ is compact, this question can be viewed as a strong form of the Hodge conjecture on $Y$: {\bf every}  rational class  is a linear combination of homology classes of algebraic cycles. And indeed in the compact case, the main result of \cite{BMM11} provides a positive answer to both Question \ref{Q}  and the Hodge conjecture as long as  $r<\frac{p+1}{3}$. The proof is of automorphic nature. There are two steps: we first show that cohomology classes obtained by the theta lift of Kudla-Millson (and which are related to special cycles by the theory of Kudla-Millson) exhaust all the cohomology classes that can be constructed using general theta lift theory. Next we use Arthur's endoscopic classification of automorphic representations of orthogonal groups to show that all cohomology classes can be obtained by theta lifting.

When $Y$ is non-compact, one expect a similar surjectivity theorem to hold for certain low degree Hodge classes of $Y$. Indeed, before \cite{BMM11}, Hoffman and He considered the case $p=3$ in \cite{HH12}. In their situation, $Y$ is a smooth Siegel modular threefold and they prove that $\Pic(Y)\otimes \CC \cong H^{1,1}(Y)$ is generated by Humbert surfaces. In the present paper, one of our goals is to extend \cite{BMM11} to all non-compact Shimura varieties of orthogonal type: 

\begin{theorem}\label{HC-thm} 
Assume that $Y$ is a connected Shimura variety associated to $\SO(p,2)$. If $r < \frac{p+1}{3}$, any cohomology class in $H^{2r}(Y,\QQ) \cap H^{r,r} (Y , \CC)$ is a linear combination (with rational coefficients) of  classes of special cycles. 
\end{theorem}

As in the case of \cite{BMM11}, our proof relies on Arthur's classification \cite{Ar13} which depends on the stabilization of the trace formula for disconnected groups. The ordinary trace formula has been established and stabilized by Arthur, but the stabilized {\it twisted} trace formula was unavailable until recently. In a series of recent work (cf.~\cite{Wa13a,Wa13b,Wa13c}),  Moeglin and Waldspurger give a proof of the stabilized twisted trace formula for disconnected groups. This means that our results are now unconditional.
 
\subsection{General results for arithmetic manifolds} As we will see later, Theorem \ref{HC-thm} actually follows from a more general result. Let $\AA$ be the ring of adeles of $\QQ$. Let $G=\SO(V)$ be the special orthogonal group associated to a quadratic space $V$ over $\QQ$. Assume that $V$ has signature $(p,q)$ over $\RR$. Then $G(\RR )=\SO(p,q)$ and the quotient 
$$D=\SO_0(p,q)/\SO(p)\times \SO(q)$$ 
is a symmetric space, where  $\SO_0(p,q)$ is the connected component of $G(\RR )$ containing the identity. 
Fix a neat compact open subgroup $K\subseteq G(\AA_f)$, then $Y_K=\Gamma_K\backslash D$ is a smooth congruence arithmetic manifold associated to $G$, where $\Gamma_K=G(\QQ )\cap K$ is a congruence arithmetic subgroup of $G(\QQ )$. The special cycles on $Y_K$ are defined as finite linear combinations of arithmetic submanifolds of the same type, see \S \ref{S8}.

Let us consider the space $\bar{H}^{nq}(Y_K,\CC)$ of square integrable harmonic $nq$-forms in $H^{nq}(Y_K,\CC)$. 
We recall from \cite{BMM11} that the decomposition of exterior powers of the cotangent bundle of $D$ under the action of the holonomy group, i.e. the maximal compact subgroup of $G$, yields a natural notion of {\it refined Hodge decomposition} of the cohomology groups of the associated locally symmetric spaces. Let $\mathfrak{g} = \mathfrak{k} \oplus \mathfrak{p}$ be the (complexified) Cartan decomposition of $G(\RR )$ associated to some base-point in $D$. 
As a representation of $\SO(p,\CC) \times \SO (q, \CC)$ the space $\mathfrak{p}$ is isomorphic to $V_+ \otimes V_-^*$ where $V_+= \CC^p$ (resp. $V_- = \CC^q$) is the standard representation of $\SO(p, \CC)$ (resp. $\SO(q, \CC)$). 
The refined Hodge types correspond to irreducible summands in the decomposition of $\wedge^{\bullet} \mathfrak{p}^*$ as a $(\SO(p,\CC) \times \SO (q, \CC))$-module.
In the case of the group $\mathrm{SU}(n,1)$ (then $D$ is the complex hyperbolic space) it is an exercise to check that one recovers the usual Hodge-Lefschetz decomposition. In general the decomposition is much finer. In our orthogonal case, 
it is hard to write down the 
full decomposition of $\wedge^{\bullet} \mathfrak{p}$ into irreducible modules. Note that, as a $\GL (V_+) \times \GL (V_-)$-module, the decomposition is already quite complicated. We have (see \cite[Equation (19), p. 121]{Fulton}):
\begin{equation} \label{GLdec}
\wedge^R (V_+ \otimes V_-^*) \cong \bigoplus_{\mu \vdash R} S_{\mu} (V_+) \otimes S_{\mu^*} (V_-)^*.
\end{equation}
Here we sum over all partition of $R$ (equivalently Young diagram of size $|\mu|=R$) and $\mu^*$ is the conjugate partition (or transposed Young diagram). 

Since $\wedge^{\bullet} \mathfrak{p} = \wedge^{\bullet} ( V_+ \otimes V_ -^*)$, the group $\mathrm{SL}(q) = \mathrm{SL}(V_ -)$ acts on $\wedge^{\bullet} \mathfrak{p^*}$. 
We will be mainly concerned with elements of $(\wedge^{\bullet} \mathfrak{p^*} )^{\SL (q)}$ --- that is elements that are trivial 
on the $V_-$-side. In general $(\wedge^{\bullet} \mathfrak{p^*} )^{\SL (q)}$ is {\it strictly} contained in $(\wedge^{\bullet} \mathfrak{p^*} )^{\SO (q)}$.  If $q$ is even there exists an invariant element 
$$e_q \in (\wedge^{q} \mathfrak{p^*} )^{\SO (p  ) \times \SL (q) },$$ 
the {\it Euler class/form}. We define $e_q=0$ if $q$ is odd. 

The subalgebra $\wedge^{\bullet} (\mathfrak{p}^*)^{\mathrm{SL}(q)}$ of $\wedge^{\bullet} (\mathfrak{p}^*)$ is invariant under $K_{\infty}= \SO ( p ) \times \SO (q)$. 
Hence, we may form the  associated subbundle 
$$F= D \times_{K_{\infty}} (\wedge^{\bullet} \mathfrak{p}^*)^{\mathrm{SL}(q)}$$ 
of the bundle 
$$D \times_{K_{\infty}} (\wedge^{\bullet}  \mathfrak{p}^*)$$
of exterior powers of the cotangent bundle of $D$. The space of sections of $F$ is invariant under the Laplacian
and hence under harmonic projection, compare \cite[bottom of p. 105]{Chern};  it is a subalgebra of the algebra of  differential forms.
 
We denote by $\bar{H}^{\bullet} (Y )^{\rm SC}$ the corresponding subspace of $\bar{H}^{\bullet} (Y )$. When $q=1$ we have $\bar{H}^{\bullet} (Y )^{\rm SC} = \bar{H}^{\bullet} (Y )$ and when $q=2$ we have 
$$\bar{H}^{\bullet} (Y )^{\rm SC} = \oplus_{r=1}^p \bar{H}^{r,r} (Y ).$$

Codimension $r$ special cycles yield classes that belong to the subspace $\bar{H}^{rq} (Y )^{\rm SC}$ and we have:
\begin{equation} \label{subringSC}
\bar{H}^{\bullet} (Y )^{\rm SC} = \oplus_{t=0}^{[p/2]} \oplus_{k=0}^{p-2t} e_q^{k} \bar{H}^{t \times q}_{\mathrm{cusp}} (Y),
\end{equation}
(see \cite{BMM11}). By analogy with the usual Hodge-Lefschetz decomposition,  we call $\bar{H}^{r \times q} (Y)$ 
the {\it primitive part} of $\bar{H}^{rq} (Y)^{\rm SC}$.
We see then that if $q$ is odd the above special classes have {\it pure refined Hodge type} and if $q$ is even each such class is the sum of at most $r+1$
refined Hodge types. In what follows we will consider the primitive part of the special cycles i.e. their projections into the subspace associated to the refined Hodge type $r \times q$.
Then our main result is:
\begin{theorem}\label{main-theorem}
Let $Y_K$ be a connected arithmetic manifold associated to $\SO(V)$. Then the subspace $\bar{H}^{r\times q}(Y_K,\CC)$ is spanned by the Poincar\'e dual of special cycles with rational coefficients for $r<\min\{\frac{p+q-1}{3}, \frac{p}{2} \}$.
\end{theorem}
\begin{remark}
The bound $\frac{p+q-1}{3}$ is conjectured to be the sharp bound; see \cite{BMM11} for some evidences. 
\end{remark}

In small degree it follows from works of Zucker that $\bar{H}^\bullet (Y) = H^\bullet (Y)$, we review these results in Section 6. 
If $q=1$, $Y_K$ is a congruence hyperbolic manifold and the special cycles on $Y_K$ are linear combinations of totally geodesic submanifolds. We obtain
\begin{corollary}
Let $Y$ be a smooth hyperbolic manifold of dimension $p$. Then for all $r< \frac{p}{3}$ the $\QQ$-vector space $H^r (Y,\QQ)=\bar{H}^r(Y,\QQ)$ is spanned by classes of totally geodesic submanifolds of codimension $r$. 
\end{corollary} 

As in \cite{BMM11}, our proof of Theorem \ref{main-theorem} is also valid for cohomology groups with non-trivial coefficients (cf.~\cite[Theorem 1.13]{BMM11}). The argument is basically the same but would require much more notation. For the ease of the reader, we shall restrict ourselves to the case of trivial coefficients.

\subsection{Outline} In Section 3 we briefly review Arthur's endoscopic classification of automorphic representations of orthogonal groups. Then in Section 4 we recall how the latter was used in \cite{BMM11} to prove that automorphic representations that contribute to the low degree cohomology have a very special Arthur parameter (see Theorem \ref{nontemper}). These representations occur in the discrete $L^2$ automorphic spectrum and as such are either cuspidal or residual. For cuspidal ones, it was
proved in \cite{BMM11} that having this ``very special Arthur parameter'' forces them to occur as $\theta$-lifts. One main new ingredient of this paper is the extension of this latter result to {\it residual} representations. In Section 5 we first show that residual representations with this ``very special Arthur parameter'' occur as residues of very particular Eisenstein series (see Proposition \ref{noncuspidalrep}). The results of Section 4 and 5 are then used in Section 6 to prove that {\it all} (both cuspidal and residual) automorphic representations that contributes to the low degree cohomology occur as $\theta$-lifts (see Theorem \ref{surjoftheta}). We finally prove our main result in Section 7 and 8. Work of Borel and Zucker allows to reduce to $L^2$-cohomology. We then mainly use theory of Kudla-Millson.

\subsection{Acknowledgements}  
The second author has benefited greatly from discussions with S.S. Kudla and Zhiwei Yun.  He also would like  to thank Brendan Hassett, Jun Li and  Radu Laza for helpful conversations and comments.

\section{Notations and Conventions}

\subsection{General notations} Throughout this paper, let $\AA$ be the adele ring of  $\QQ$.  We write $\AA_f$ for its finite component respectively.   We denote by  $|\cdot|_{p}$ the absolute value on local fields $\QQ_p$  and  $|\cdot |_\AA=\prod_p |\cdot|_{p}$  the absolute value of adelic numbers in $\AA$. If $G$ is a classical  group over $\QQ$ we let $G(\AA)$ be the group of its adelic points and $X(G)_\QQ$ be the group of characters of $G$ which are defined over $\QQ$ 
 
\subsection{Automorphic forms} 
Let $G(\AA)^1=\{x\in G(\AA):|\chi(x)|=1, \chi\in X(G)_\QQ\}$ be the subgroup of $G(\AA)$. The quotient $X_G:=G(\QQ)\backslash G(\AA)^1$ has finite volume with respect to the natural Haar measure. Let $L^2(X_G)$ be the Hilbert space of square integrable functions on $X_G$. According to Langlands' spectral decomposition theory, the discrete part $L^2_{\rm dis}(X_G)$ of $L^2(X_G)$ 
$$L^2_{\rm dis}(X_G)=\widehat{\bigoplus} m_{\rm dis} (\pi)\pi$$
is the Hilbert sum of all irreducible $G(\AA)$-submodules $\pi$ in  $L^2(X_G)$ with finite non-zero multiplicity $m_{\rm dis} (\pi)$. We shall denote by $\mathcal{A}_2 (G)$ the set of those representations and by $\mathcal{A}_c (G)$ the subset of those that are cuspidal.

\subsection{Orthogonal groups}
Here we fix some notations for orthogonal groups.
Let $V$ be a non-degenerate quadratic space of dim $m$ over $\QQ$ and let $G=\SO(V)$ be the corresponding special orthogonal group. Set $N=m$ if $m$ is even and $N=m-1$ if $m$ is odd. Throughout this paper, we denote by $\SO(n)$ the special orthogonal group associated to 
$$J= \left(
\begin{array}{ccc}
0 & & 1 \\
 & \iddots & \\
1 &  & 0 
\end{array} \right).$$ 
The group $G=\SO(V)$ is an inner form of a quasi-split form $G^\ast$, where $G^\ast$ is the odd orthogonal group $\SO(m)$ when $m$ is odd
or the outer twist $\SO(m,\eta)$ of the split group $\SO(m)$ when $m$ is even. Let $\theta_N$ be the automorphism of $\GL(N)$:
$$\theta_N(g)=J\cdot ~^tg^{-1}\cdot J^{-1}, ~g\in \GL(N)$$
and let $\widetilde{\GL}(N)=\GL(N)\rtimes \left<\theta_N\right>$.  The complex dual of  $G^\ast$ is $G^{\vee} = \Sp(N,\CC)$ if $m$ is odd and $G^{\vee} = \SO (N, \CC)$ if $m$ is even.
In the first case the $L$-group of $G$ is ${}^L G=G^\vee\times \Gamma_\QQ$ with $\Gamma_{\QQ}={\rm Gal}(\bar{\QQ}/\QQ)$, whereas, in the second case, we have ${}^L G = G^{\vee} \rtimes \Gamma_\QQ$, where $\Gamma_\QQ$ acts on $G^{\vee}$
by an order 2 automorphism --- trivial on the kernel of $\eta$ --- and fixes a splitting, see \cite[p. 79]{Asterisque} for an explicit description.

We finally denote by $\Pi \mapsto \Pi^{\theta}$ the action of the automorphism $\theta= \theta_N$ on representations of $\GL(N)$.

\section{Arthur's classification theory}
In this section, we review Arthur's theory  of classification of   square integrable automorphic representations of special orthogonal groups.  

\subsection{Preliminaries} Let $G$ be an algebraic  group over $\QQ$. Any $\pi\in \cA_2(G)$ has  a  unique decomposition
$$\pi=\otimes_v \pi_v,$$
where $\pi_v$ is an admissible representation of $G(\QQ_v)$ over all places. Moreover,  $\pi_v$ is unramified for all $v$ outside a finite set  $S$ of places, see \cite{Flath}.  Langlands and Arthur theories investigate how the discrete automorphic representations are assembled from the local representations.  

We first recall that the local Langlands classification says that  the set of admissible representations of $G(\QQ_v)$ are parametrized by equivalence classes of morphisms,
\begin{equation}
\phi_v : L_{\RR}\longrightarrow  {}^L\!G. 
\end{equation}
where $L_{\QQ_\infty} = W_{\RR}$, if $v=\infty$, and $L_{\QQ_p}= W_{\QQ_p} \times \mathrm{SU} (2)$, if $v$ is a (finite) prime $p$, in any case $W_{\QQ_v}$ is the local Weil group of $\QQ_v$ and $ {}^L\!G $ the complex dual group. We call $\phi_v$ a Langlands parameter and denote by $\prod (\phi_v )$ the packet of admissible representations associated to $\phi_v$. When $G=\GL(N)$, the local Langlands packet $\prod (\phi_{v} )$ contains exactly one element.

\subsection{Formal parameters of orthogonal groups} 
Arthur \cite{Ar13} parametrizes discrete automorphic representations of $G$ by some set of formal sums of formal tensor products (called global Arthur parameter) 
\begin{equation}\label{globalArthurpacket}
\Psi = \boxplus(\mu_j\boxtimes R_{a_j}),
\end{equation}
where each $\mu_j$ is an irreducible, unitary, cuspidal automorphic representation of $\GL(d_j)$, $R_{a_j}$ denotes the irreducible representation of $\SL(2,\CC)$ of dimension $a_j$, $\sum a_jd_j=N$, and we have: 
\begin{enumerate}
\item the $\mu_j\boxtimes R_{a_j}$ are all distinct, and
\item for each $j$, we have $\mu_j^{\theta} = \mu_j$. 
\end{enumerate}
  
To each such global Arthur parameter $\Psi$, we associate an irreducible representation $\Pi_{\Psi}$ of the general linear group $\GL(N, \AA)$: the representation $\Pi_\Psi$ is the induced representation
\begin{equation}\label{transfer}
\ind(\Pi_{\mu_1,a_1}\otimes \Pi_{\mu_2,a_2}\ldots \otimes\Pi_{\mu_\ell,a_\ell}),
\end{equation} 
where the inducing parabolic subgroup of $\GL(N)$ has Levi factor $\GL(d_1a_1)\times\cdots \times \GL(d_\ell a_\ell)$.  Here $\Pi_{\mu_j,a_j}$ is the unique Langlands quotient of the normalized parabolically induced representation
\begin{equation}\label{repGL}
\ind(\mu_j|\det|_\AA^{\frac{a_j-1}{2}}\otimes\mu_j|\det|_\AA^{\frac{a_j-3}{2}}\otimes \ldots \otimes\mu_j|\det|_\AA^{\frac{1-a_j}{2}}).
\end{equation} 
from a parabolic subgroup of $\GL(d_ja_j)$ whose Levi factor is the group of block  diagonal matrices $\GL(d_j)\times \ldots\times\GL(d_j)$ ($a_j$-copies). 

\subsection{} 
Let $k=\QQ_v$ with $v$ finite or infinite. Arthur has duplicated the construction from the global case to define the local parameters of admissible representations of $G(k)$. More precisely, one can define the local Arthur parameter of $G(k)$ to be the formal symbol 
\begin{equation}\label{localparameter}
\Psi_k=\boxplus(\mu_j\boxtimes R_{a_i})
\end{equation}
where $\mu_i$ is a tempered, $\theta$-stable, irreducible representation of $\GL(d_i,k)$ that is square integrable modulo the center and $R_{a_i}$ is the $a_i$-dimensional irreducible representation of $\SL(2,\CC)$. 
Similar as the global construction \eqref{repGL}, it determines a unique irreducible representation $\Pi_{\Psi_k}$ of $\GL(N,k)$. 
 
By local Langlands correspondence, the local Arthur parameter $\Psi_k$ can be represented by a continuous homomorphism
\begin{equation}
\Psi_k:L_{k}\times \SL_2(\CC)\longrightarrow {}^L G,
\end{equation}
which has bounded image on its restriction to $L_{k}$ and is algebraic on $\SL(2,\CC)$. By abuse of notation, we still use $\Psi_k$ to denote this homomorphism.  One attach an $L$-map $\phi_{\Psi_k}:L_k\rightarrow {}^L G$ to $\Psi_k$ by composition, i.e. 
\begin{equation} \label{eq:comp}
\phi_{\Psi_k}(w)=\Psi_{k}\left(w, \left(\begin{array}{cc}
 |w|^{1/2}&  \\ 
 & |w|^{-1/2}
\end{array} \right) \right).
\end{equation}
Note that if $G$ is not quasi-split, the map $\phi_{\Psi_k}$ might not be relevant (in Langlands' sense) and therefore does not define an $L$-parameter in general. 
However outside a finite set of places $\phi_{\Psi_k}$ indeed defines a (relevant) $L$-parameter.  

\begin{remark}
As in \cite{BMM11}, here  we ignore the complication of lack of generalized Ramanujan conjecture  just for simplicity of notations. In general setting, we need to introduce a larger set of homomorphisms without the boundedness condition. Then the approximation to Ramanujan's conjecture in \cite{LRS99} is enough for our purpose (see also \cite{BC13} $\S$3.1). 
\end{remark}

\subsection{Arthur's packets}
Given a local Arthur parameter $\Psi_k$, Arthur has associated a finite packet $\prod (\Psi_k)$ of representations of $G(k)$ with multiplicities to each $\Psi_k$ using the twisted Langlands-Shelstad transfer. The details of transfer are not important for us, so we omit the description and only recall the relevant properties of Arthur local packets. We will first describe these properties when $G$ is quasi-split.

\subsection{} Suppose that {\bf $G=G^*$ is quasi-split.} 
Let $\pi$ be a unitary representation of $G=G(k)$. By the local Langlands' classification \cite[Theorem XI.2.10]{BW00} there exist `Langlands data' $P$, $\sigma$, $\lambda$, with 
$\lambda \in (\mathfrak{a}_P^*)^+$, such that $\pi$ is the unique quotient of the corresponding standard induced module. The parabolic subgroup $P$ comes with a canonical embedding of its dual 
chamber $(\mathfrak{a}_P^*)^+$ into the closure $\overline{(\mathfrak{a}^*_{B})^+}$ of the dual chamber of the standard Borel subgroup $B$ of $\GL (N)$. We shall write 
$\Lambda_\pi$ for the corresponding image of the point $\lambda$; and call it the {\it exponent} of the representation $\pi$. One writes 
$$\Lambda ' \leq \Lambda , \quad \Lambda' , \Lambda \in \overline{(\mathfrak{a}^*_{B})^+},$$
if $\Lambda - \Lambda '$ is a nonnegative integral combination of simple roots of $(B , A_B)$. This determines a partial order. 

The key property of local Arthur packets that we will use is the following proposition that follows from \cite[Theorem 2.2.1 and Eq. (2.2.12)]{Ar13}. By local Langlands correspondence, each unitary 
representation $\pi$ of $G(k)$ belongs to a unique $L$-packet $\prod (\phi)$ associated to an $L$-parameter $\phi : L_k\rightarrow {}^L G$. Moreover: if $\pi$ and $\pi' $ both belong to $\prod (\phi)$ then $\Lambda_\pi = \Lambda_{\pi'}$, we may therefore define $\Lambda_\phi$ as their common value. 

\begin{proposition} \label{prop:A1}
Let $\Psi_k$ be a local Arthur parameter. Then the associated $L$-packet $\prod (\phi_{\Psi_k})$ is included in the Arthur packet $\prod (\Psi_k)$. Moreover: if $\pi \in \prod (\Psi_k)$ then 
$$\Lambda_\pi \leq \Lambda_{\phi_{\Psi_k}}.$$
\end{proposition}

Note that a global Arthur parameter $\Psi$ gives rise to a local Arthur parameter $\Psi_{v}$ at each place $v$. The global part of Arthur's theory (see \cite[Theorem 1.5.2]{Ar13}) then implies:

\begin{proposition} \label{prop:A2}
Let $\pi$ be an irreducible automorphic representation of $G(\AA)$. Then there exists a global Arthur parameter $\Psi$ such that $\pi_v\in \prod(\Psi_v)$ at each place $v$. Moreover, 
there exists a finite set $S$ of places of $\QQ$ containing all infinite places such that for all $v\notin S$, the representation $\pi_v$ is unramified and its $L$-parameter is $\phi_{\Psi_v}$, in particular $\pi_v$ belongs to the $L$-packet $\prod (\phi_{\Psi_v})$.
\end{proposition}

\subsection{} Suppose now that {\bf $G$ is not quasi-split}. Recall the map $\phi_{\Psi_k}$ associated to a local Arthur parameter $\Psi_k$ might not be relevant (in Langlands' sense) now and therefore might not define a local $L$-packet. In particular Proposition \ref{prop:A1} does not make sense. Arthur stable trace formula however allows to compare the discrete automorphic spectrum of $G$ with that of $G^\ast$.
This is the subject of work in progress of Arthur, the results being announced in \cite[Chapter 9]{Ar13}. We will only need weak version of them that can be directly deduced from what Arthur has already written in the quasi-split case. It first follows from \cite[Proposition 9.5.2]{Ar13} that:
\begin{enumerate}
\item it corresponds to any irreducible automorphic representation $\pi$ of $G(\AA)$ a global Arthur parameter $\Psi$, and
\item at each place, one can still define a local packet $\prod(\Psi_{v})$ of irreducible unitary representations of $G(\QQ_v)$ (via the stable trace transfer),
\end{enumerate}
in such a way that Proposition \ref{prop:A2} still holds when $S$ is chosen so that for all $v \notin S$, the group $G(\QQ_v) = G^* (\QQ_v)$ is quasi-split. We subsume this in the following

\begin{proposition} \label{prop:A2bis}
Let $\pi$ be an irreducible automorphic representation of $G (\AA)$ which occurs (discretely) 
as an irreducible summand in $L^2 (G (\QQ) \backslash G (\AA))$. Then there exist a global
Arthur parameter $\Psi$ and a finite set $S$ of places of $\QQ$ containing all Archimedean ones 
such that for all $v \notin S$, the group $G(\QQ_v)=G^* (\QQ_v )$ is quasi-split, the representation $\pi_v$ is unramified and the $L$-parameter of $\pi_v$ is $\phi_{\Psi_v}$. 
\end{proposition}

We now explain how to replace Proposition \ref{prop:A1}. Let $v$ be a place of $\QQ$ and let $k=\QQ_v$. By composition \eqref{eq:comp} the local Arthur parameter $\Psi_k$ still defines an $L$-parameter 
 for $\mathrm{GL}(N)$. Let $\Lambda_{\Psi_k}$ be its exponent. Note that if $G(k)=G^* (k)$ is quasi-split case we have:
$$\Lambda_{\Psi_k} = \Lambda_{\phi_{\Psi_k}}.$$
Now recall that if $\pi$ is a unitary representation of $G=G(k)$ then $\pi$ is the unique quotient of some
standard module $\mathrm{St}(\pi)$ associated to a Langlands datum. Here again the inducing parabolic subgroup comes with a canonical embedding of its dual 
chamber into the closure of the dual chamber of the standard Borel subgroup $B$ of $\GL (N)$. Therefore the exponent
$\Lambda_\pi$ still makes sense as an element of $\overline{(\mathfrak{a}^*_{B})^+}$. 

Now following \cite[pp. 66--72]{Ar13} we may replace Proposition \ref{prop:A1} by the following proposition (see also \cite[Appendix A]{BMM13} and \cite[Appendix]{BMM11} for similar results and more details). 

\begin{proposition} \label{prop:A1bis}
Let $k=\QQ_v$ with $v$ finite or not. For every irreducible unitary representation $\pi$ of $G(k)$ that belongs to a local Arthur packet $\prod (\Psi_k)$ we have:
$$\Lambda_\pi \leq \Lambda_{\Psi_k}.$$
\end{proposition}

\begin{remark}
1. Due to our particular choice of groups here 
$$\mathrm{St} (\pi) = \mathrm{ind} \left( \otimes_{(\rho , a ,x)} \mathrm{St} (\rho , a) | \cdot |^{x} \otimes \sigma \right)$$
where each triplet $(\rho , a ,x)$ consists of a unitary cuspidal representation $\rho$ of a linear group, a integer $a$ and positive real number $x$, $\mathrm{St} (\rho , a)$ is the corresponding Steinberg representation --- the unique irreducible sub-representation of the representation induced from $\rho |\cdot |^{(a-1)/2} \otimes \ldots \otimes \rho |\cdot |^{(1-a)/2}$ --- and $\sigma$ is a tempered representation of an orthogonal group of the same type as $G(k)$ but of smaller dimension. The set of $x$'s with multiplicities is the same as the real part of the exponent of $\pi$. 

2. Let $\phi$ be the $L$-parameter of $\pi$. By embedding ${}^L G$ into ${}^L \mathrm{GL} (N)$, we may associate to $\phi$ an irreducible unitary representation of $\mathrm{GL} (N,k)$. Arthur's proof of \cite[Proposition 3.4.3]{Ar13} in fact shows that the latter occurs as an (irreducible) sub-quotient of the local standard representation associated to the local Arthur parameter $\Psi_k$. 
\end{remark}

\subsection{Automorphic $L$-functions} \label{S27}
Let $\pi \in \cA_2(G)$ and let $\Psi=\boxplus(\mu_i\boxtimes R_{a_i})$ be the associated global Arthur parameter. We refer to \cite{Borel} for the definition of the local $L$-factor of an unramified local admissible representation. 
Let $S$ be a finite set of places that contains the finite set of Proposition \ref{prop:A2bis} and the set of places where some $(\mu_i)_v$ is ramified. We can then define the formal Euler product
$$L^S (s,\Pi_{\Psi})=\prod_{j} \prod_{v \notin S} L(s-\frac{a_j-1}{2},\mu_{j,v}) \times \cdots \times L(s-\frac{1-a_j}{2},\mu_{j,v})$$
and the partial $L$-function $L^S(s,\pi)$ as the formal Euler product $\prod\limits_{v\notin S}L (s,\pi_v)$. It then follows from Proposition \ref{prop:A2bis} that we have 
$$L^{S}(s,\pi)=L^{S}(s,\Pi_\Psi).$$
Note that according to Jacquet and Shalika \cite{JacquetShalika} this partial $L$-function is absolutely convergent for Re$(s) \gg 0$ and extends to a meromorphic function of $s$. 

For a self-dual automorphic character $\eta$,  we can similarly write $L^S(s,\eta\times \pi)$ as the partial $L$-function associated to the parameter $\boxplus ((\eta\otimes \mu_i)\boxtimes R_{a_i})$. 
Similarly, we can define the $L$-function $L(s,\Pi_\Psi,r_G)$ attached to $\Psi$, where $r_G$ is the finite dimensional representation $\mathrm{Sym}^2$ if $m$ is odd and $\wedge^2$ if $m$ is even.

\subsection{}
As in \cite{BMM11}, we shall say that an automorphic representation $\pi\in \cA_2(\SO(V))$ is {\it highly non-tempered} if its global Arthur parameter contains a factor $\eta\boxtimes R_a$, where $\eta$ is a quadratic character and $3a>m-1$. 
Assuming further that $\pi$ has a regular infinitesimal character, the poles of the partial $L$-function of a highly non-tempered representations  can be easily determined using the following lemma. 

\begin{lemma} \label{L43}
Suppose $\pi\in\cA_2(\SO(V))$ is highly non-tempered with corresponding factor $\eta \boxtimes R_a$ in its global Arthur parameter $\Psi$. If $\pi$
has a regular infinitesimal character, then $\Psi$ must have the form:
\begin{equation}\label{infinitesimalcharacter}
\Psi =(\boxplus_{(\tau,b)}\tau\boxtimes R_b) \boxplus\eta\boxtimes  R_a,
\end{equation}
where each pair $(\tau,R_b)$ satisfies the condition either $b<a$ or $b=a$ and $\tau \neq \eta$, and $\sum\limits_{(\tau,b)} bd_\tau+a=N$. 
Moreover, the partial $L$-function $L^S(s,\Pi_{\Psi})$ is holomorphic for ${\rm Re}(s)>\frac{a+1}{2}$ and has a simple pole at $s=\frac{a+1}{2}$.
\end{lemma}
\begin{proof}
See \cite{BMM11} Lemma 4.3.
\end{proof}

\section{Cohomological unitary representations for orthogonal groups} 
In this section, we review the cohomological representations of classical groups and  determine the (local) Arthur parameter of square integrable automorphic representations of $\SO(V)(\AA)$ with cohomology at infinity. The local parameters have been computed in \cite{BMM11}. We refer the readers to \cite[Section 4 and Appendix]{BMM11} for more details.

\subsection{Cohomological $(\frg,K_\RR)$-module} 
Let $G_\RR$ be a real connected reductive group with finite center. Let $K_\RR\subset G_\RR$ be a maximal compact subgroup and denote by $\theta$ the associated Cartan involution. 
Let $\frg_0=\frk_0+\frp_0$ be the corresponding Cartan decomposition of $\frg_0=\mathrm{Lie}(G_\RR)$. We write $\fra=(\fra_0)_\CC$ for the complexification of a real Lie algebra $\fra_0$.

A unitary representation $(\pi_\RR,V_{\pi_\RR})$ is {\it cohomological} if the associated  $(\frg,K_\RR)$-module $(\pi_\RR^\infty,V^\infty_{\pi_\RR})$ has nonzero relative Lie algebra cohomology, i.e. 
$$H^\bullet(\frg,K_\RR;V_{\pi_\RR}^\infty )\neq 0.$$

\subsection{}
The unitary cohomological $(\frg,K_\RR)$-modules have been classified by Vogan and Zuckerman \cite{VZ84} as follows:
let $i\frt_0\subseteq\frk_0$ be a Cartan subalgebra of $\frk_0$. For $x\in \frt_0$, let $\frq$ be the sum of nonnegative eigenspaces of $\mathrm{ad}(x)$. 
Then $\frq$ is a $\theta$-stable parabolic subalgebra of $\frg$ with a $\theta$-stable Levi decomposition 
$$\frq=\frl+\fru.$$ 
The normalizer of $\frq$ in $G$ is the connected Levi subgroup $L\subseteq G$ with Lie algebra $\frl_0=\frl\cap\frg_0$. 
Via cohomological induction $\frq$ determines a unique irreducible unitary representation $A_{\frq}$ of $G$. 
More precisely, let $\frh\supseteq\frt$ to a Cartan subalgebra $\frh$ of $\frg$ and we fix a positive system $\Delta^+(\frh,\frg)$ of roots of $\frh$ in $\frg$. 
Let $\rho$ be the half sum of roots in $\Delta^+(\frh,\frg)$ and $\rho(\fru\cap \frp)$ the half sum of roots of $\frt$ in $\fru\cap\frp$. Denote by $\mu(\frq)$ the irreducible representation of $K_\RR$ of highest weight $2\rho(\fru\cap \frp)$.  
Then we have
\begin{proposition}\label{VZclassification} 
The module $A_\frq$ is the unique irreducible unitary $(\frg,K_\RR)$-module such that:
\begin{enumerate}
\item $A_\frq $ contains the $K_\RR$-type $\mu(\frq)$ occurring with multiplicity one.
\item  $A_\frq $ has the  infinitesimal character  $\rho$.
\end{enumerate}
Moreover, let $(\pi_\RR, V_{\pi_\RR})$ be an irreducible unitary $(\frg,K_\RR)$-module such that 
$$H^\ast(\frg, K_\RR; V_{\pi_\RR}^{\infty})\neq 0.$$ 
Then there is a $\theta$-stable parabolic subalgebra $\frq=\frl\oplus\fru$ of $\frg$ so that $\pi_\RR\cong A_\frq$ and 
\begin{equation}\label{cohomologicalrep}
H^i(\frg,K; V_{\pi})\cong \Hom_{\frl\cap \frk}( \wedge^{i-\dim(\fru\cap\frp)}(\frl\cap\frp),\CC).
\end{equation}
\end{proposition}

Note that the isomorphism class of $A_\frq$ only depends on the intersection $\mathfrak{u} \cap \mathfrak{p}$, we may furthermore choose $\mathfrak{q}$ such that the Levi subgroup $L$ associated to $\mathfrak{l}$ has no non-compact (non-abelian) factor.

\subsection{}
Now let $G_\RR=\SO_0(p,q)$ and fix a maximal compact subgroup $K_\RR \cong \SO(p)\times \SO(q)$. 
The unitary $(\frg,K_\RR)$-modules with nonzero low degree cohomology are very particular (see \cite[Proposition 5.16]{BMM11}):

\begin{proposition}\label{cohorep}
Let $A_\frq$ be a cohomological $(\frg,K_\RR)$-module. Suppose that $R=\dim (\fru\cap \frp)$ is strictly less than $p+q-3$ and $pq/4$. Then either $L=C\times \SO(p-2r,q)$ with $C\subseteq K_\RR$ and $R=rq$ or $L=C\times \SO(p,q-2r)$ with $C\subset K_\RR$ and $R=rp$. 
\end{proposition}

By Proposition \ref{VZclassification}, the irreducible summands of the decomposition of $\wedge^\ast \frp$  give a refined Hodge structure of $H^\ast(\frg,K_\RR, V_\pi )$.
Recall that the space $\frp=\CC^p\otimes (\CC^q)^\ast$, where $\CC^p$ (resp.~$\CC^q$) is the standard representation of $\SO(p,\CC)$ (resp.~$\SO(q,\CC)$). Then there is a decomposition 
\begin{equation}
\wedge^R \frp=\bigoplus\limits_{\mu\vdash R} \SS_\mu (\CC^p)\otimes \SS_{\mu^\ast}(\CC^q)^\ast
\end{equation}
where $\mu$ is a partition of $R$, $\mu^\ast$ is the conjugation partition and  $\SS_\mu(\CC^p)$ is the Schur functor associated to $\mu$. Then the cohomological module $A_\frq $ with Levi subgroup $C\times \SO(p-2r,q)$ in Proposition \ref{cohorep} corresponds to the irreducible representation $\SS_{[r\times q]}(\CC^p)\otimes (\wedge^q \CC^q)^r$ of $\SO(p,\CC)\times\SO(q,\CC)$, where  $\SS_{[r\times q]}(\CC^p) $ is the harmonic Schur functor associated to the partition $(q,q\ldots,q)$ of $R=rq$ (see \cite[\S 5.11]{BMM11}).

Now let $e_q\in \Omega^q(D)$ be the Euler form defined in \cite{BMM11}; it is a $K_\RR$ invariant $q$-form on $D$ which is nonzero if and only if $q$ is even. Wedging with the $(qk)$-form $e_q^k$ induces a morphism:
\begin{equation}\label{lefschetzstructure}
e_q^k: H^{i}(\frg,K_\RR; A_{\frq}) \rightarrow H^{i+qk}(\frg,K_\RR; A_{\frq}) 
\end{equation}
which is an isomorphism when $q$ is even and $k$ small enough. We refer the readers to \cite[Proposition 5.15]{BMM11} for the description of $H^i(\frg,K_\RR;A_\frq )$.

In case $G_\RR=\SO(p,2)$, we write  $A_{r,r}$ for the cohomological $(\frg,K_\RR)$-module with corresponding Levi subgroup $\mathrm{U}(1)^r\times \SO(p-2r,2)$ for $r<\frac{p}{2}$. Then $H^{2k}(\frg,K_\RR, A_{r,r})$ has a pure Hodge structure of type $(k,k)$. In small degree, this refined Hodge structure and the isomorphism \eqref{lefschetzstructure} will recover the Hodge-Lefschetz structure on the cohomology group of connected Shimura varieties via Matsushima formula (see Section 7).

\subsection{Local parameters of cohomological representations} 
Assume $G_\RR=\SO_0(p,q)$ and let $\pi_\RR$ be a unitary representation of $G_\RR$ whose underlying $(\frg,K_\RR)$-module is equivalent  to $A_\frq$ associated to the Levi subgroup $L=\mathrm{U}(1)^r\times \SO(p-2r,q)$. A conjectural description of the local Arthur packets that can contain $\pi_\RR$ is given in \cite[\S 6.6]{BMM11}. It would follow that if $\pi_\RR$ belongs to a local Arthur packet $\prod (\Psi)$ then 
the real part of the exponent $\Lambda_\Psi$ satisfies 
\begin{equation} \label{eq:exponent1}
\mathrm{Re}(\Lambda_\Psi) \geq (m-2r-2,m-2r-4, \ldots , 0 , \ldots , 0 , \ldots , 2r+2-m).
\end{equation}
To our knowledge the proof of \eqref{eq:exponent1} is still open. 

However letting $P_0$ be the minimal standard parabolic subgroup of $G_\RR$ with Langlands decomposition $P_0={}^0\!MA_0N$, where $M\cong \SO(p-q,\RR)\times (\RR^\times)^q$ and $A$ is the maximal $\RR$-split torus of $P_0$ and using the standard labeling of simple roots: 
$$\Delta(P_{0},A_{0})=\{\epsilon_1,\epsilon_2\ldots, \epsilon_q\},$$  
we may write $\pi_\RR $ as the unique  Langlands quotient of the normalized induced representation 
\begin{equation}\label{Langlandsquotient}
{\rm ind}_{P_{0}}^G(e^\nu\otimes \sigma)
\end{equation}
(see \cite[\S 7]{VZ84}) or,  equivalently, $\pi_\RR$ is the unique irreducible submodule of the normalized induced representation ${\rm ind}_{P_0}^G(e^{-\nu}\otimes \sigma)$ where 
\begin{equation}\label{exponent}
\nu=\frac{m-2r-2}{2}\epsilon_1+ \frac{m-2r-4}{2}\epsilon_2+\ldots + \frac{m-2r-2q}{2}\epsilon_q \in\fra_0^\ast ,
\end{equation}
called the exponent of $\pi_\RR$, is in the positive chamber of the Weyl group and 
$\sigma$ is an irreducible representation on the compact group $\SO(p-q)$.\footnote{The Harish-Chandra module of $\sigma$ is determined by 
$\rho+\rho_u$ (cf.~See \cite{BW00} Chapter VI when $q=1$).} Using Langlands' notation we say that $\pi_\RR$ is a Langlands' quotient of a standard representation whose non-tempered ``exponents'' 
are 
$$(z\bar{z})^{m-2r-2}, \ (z\bar{z})^{m-2r-4}, \ \ldots , \ (z\bar{z})^{m-2r-2q}.$$ 
It follows in particular from Proposition \ref{prop:A1bis} that if $\Psi$ is a local Arthur parameter such that $\pi_\RR$ belongs to $\prod (\Psi)$ then the real part of the exponent $\Lambda_\Psi$ satisfies:
\begin{multline} \label{eq:exponent2}
\mathrm{Re} (\Lambda_\Psi)  \geq (m-2r-2,m-2r-4, \ldots , m-2r-2q , 0 , \ldots \\ \ldots , 0 , 2r+2q -m , \ldots , 2r+2-m).
\end{multline}
It is slightly weaker than the expected bound \eqref{eq:exponent1} but it will be enough for our purpose.

\subsection{Cohomological  automorphic representation for $\SO(V)$} \label{S35}
Coming back to the global case, we consider the Arthur parameters of the automorphic representations  $\pi=\otimes_v\pi_v\in\cA_2(\SO(V))$ with $\pi_{\infty}$ cohomological. Let $\Psi$ be the global Arthur parameter of $\pi$. By Proposition \ref{VZclassification}, we know that the infinitesimal character of $\pi_{\infty}$ is regular (equal to that of the trivial representation) and so is the infinitesimal character of $\pi^{\GL}$. This implies that for any two factors  $\mu\boxtimes R_a$ and $\mu'\boxtimes R_{a'}$ in $\Psi$, we have $\mu\neq \mu'$ unless $a$ or $a'=1$ (see \cite[Appendix]{BMM11}). 

Furthermore: if $\pi_{\infty}$ is the cohomological representation of $\SO(p,q)$ associated to some Levi subgroup $L=\SO(p-2r,q)\times \mathrm{U}(1)^r$, the inequality \eqref{eq:exponent2} on exponents at infinity forces $\Psi_{\infty}$ to contain a ``big'' $\SL (2, \CC)$-factor $R_a$ with $a \geq m-2r-1$. The following theorem is proved in \cite{BMM11}, see Proposition 6.9.  Here we quickly review the main steps of the proof. 

\begin{theorem}\label{nontemper} Let $\pi=\otimes_v\pi_v$ be an irreducible square integrable automorphic representation of $G(\AA)$. If $\pi_{\infty}$ is a cohomological representation of $\SO(p,q)$ associated to a Levi subgroup $L=\SO(p-2r,q)\times \mathrm{U}(1)^r$ with $p >2r$ and $m-1>3r$, then $\pi$ is highly non-tempered. 
\end{theorem}
\begin{proof} Since the $\SL (2, \CC)$-factors of an Arthur parameter are global, the global Arthur parameter $\Psi$ of $\pi$ contains some factor $\mu \boxtimes R_a$ with $a \geq m-2r-1 >\frac{m-1}{3}$. Then either $\mu$ is a character or $\mu$ is $2$-dimensional by dimension consideration. We have to exclude the latter case. If $a>(m-1)/2$,  this is automatic. In general, we first observe that the explicit description of $\pi_{\infty}$ as a Langlands' quotient not only shows $\pi_{\infty}$ has a big exponent but that this exponent is a {\it character exponent}. Then  the proof \cite[Proposition 6.9]{BMM11} (postponed to the Appendix there) shows that $\Psi_\infty$ also has a big character exponent and this excludes the possibility that $\mu$ is $2$-dimensional.
\end{proof}

\section{Residual representations} 

In this section, we study the residual spectrum of  $G=\SO(V)$ via the theory of Eisenstein series developed by Langlands (cf.~\cite{L76,MW95}). Our main purpose  is to show that non-cuspidal representations with cohomology at infinity are residues of Eisenstein series from the rank one maximal parabolic subgroup, see Proposition \ref{noncuspidalrep}. 

\subsection{Residual spectrum} The right regular representation of $G(\AA)$ acting on the Hilbert space $L^2 (G(\QQ) \backslash G(\AA))$ has continuous spectrum and discrete spectrum:
$$L^2 (G(\QQ) \backslash G(\AA)) = L^2_{\rm dis} (G(\QQ) \backslash G(\AA)) \oplus L^2_{\rm cont} (G(\QQ) \backslash G(\AA)).$$
We are interested in the discrete spectrum. Using his theory of Eisenstein series, Langlands decomposes this space as:
$$L^2_{\rm dis} (G(\QQ) \backslash G(\AA)) = \bigoplus_{(M , \sigma)} L^2_{\rm dis} (G(\QQ) \backslash G(\AA))_{(M, \sigma)},$$
where $(M,\sigma)$ is a Levi subgroup with a cuspidal automorphic representation $\sigma$ (taken up to conjugacy and whose central character restricts trivially to the center of $G$) and $L^2_{\rm dis} (G(\QQ) \backslash G(\AA))_{(M, \sigma)}$ is a space of residues of Eisenstein series associated to $(M,\sigma)$.

Note that the subspace 
$$L^2_{\rm cusp} (G(\QQ) \backslash G(\AA)):=\oplus_{(G , \pi)} L^2_{\rm dis} (G(\QQ) \backslash G(\AA))_{(G, \pi)}$$
is the space of cuspidal automorphic representations of $G$. Its orthogonal complement in $L^2_{\rm dis} (G(\QQ) \backslash G(\AA))$ is called the residual spectrum and we denote it by $L^2_{\rm res} (G(\QQ) \backslash G(\AA))$. In order to understand it we have to introduce the relevant Eisenstein series. 

\subsection{Eisenstein series} Let $P=MN$ be a maximal standard parabolic subgroup of $G$ and let $(\sigma ,V_\sigma)$ be an irreducible automorphic representation of $M(\AA)$. Choosing a section $f_s$ of the parabolically induced representation 
$$\rmI ( s, \sigma ) = \mathrm{ind}_{P(\AA)}^{G(\AA)} (\sigma | \det |^s),$$
where $|\cdot |$ is the adelic absolute value, and the induction is normalized,\footnote{Note that, since $P$ is maximal, the group of characters of $M$ is one dimensional generated by $\det$.} we define the Eisenstein series by the analytic continuation from the domain of convergence of the series
\begin{equation}\label{eq:eis}
E(f_s , g )=\sum\limits_{\gamma\in P(\QQ)\backslash G(\QQ)} f_s(\gamma g).
\end{equation}
When $\sigma$ is cuspidal, Langlands (see \cite[Sect. IV.1]{MW95}) shows that the Eisenstein series \eqref{eq:eis} is a meromorphic function of $s$ with finitely many poles, all simple, in the region $\mathrm{Re}(s)>0$.  Then 
$L^2_{\rm dis} (G(\QQ) \backslash G(\AA))_{(M, \sigma)}$ is the space spanned by the residues $((s-s_0)E(f_s, g))_{s=s_0}$ of the Eisenstein series \eqref{eq:eis} at a simple pole $s=s_0$ when varying the sections $f_s$ in $V_\sigma$. 

\subsection{Parabolic subgroups} The maximal parabolic subgroups of $G=\SO(V)$ can be described as below: let 
$$V=V_t+U_{t,t},$$
be the Witt decomposition of $V$, where $V_t$ is anisotropic and $U_{t,t}$ is $t$-copies of hyperbolic planes spanned by $u_i, u'_i$ for $i=1,2\ldots,t$ with $(u_i,u_j')=\delta_{ij}$. The integer $t$ is the Witt index of $V$. Fixing an anisotropic quadratic space $V_t$,  for any $0\leq d\leq t$, we  write 
$$V=U_{d}+V_{d}+U'_{d},$$ 
where $U_d,U'_d$ are  subspaces spanned by $u_1,\ldots u_d$ and $u_1',\ldots,u_d'$ respectively, and $V_t \subset V_d$.
The subgroup $P\subseteq \SO(V)$  which stabilizes $U'_d$  is a standard parabolic subgroup over $\QQ$ with Levi decomposition
$P=MN$, where 
\begin{equation}\label{parabolic}
M\cong \GL(U_{d})\times \SO(V_{d}).
\end{equation} 

Let $W(M)$ be the set of elements $w$ in the Weyl group of $G$, of minimal length in their left coset modulo the Weyl group of $M$, such that $wMw^{-1}$ is the Levi factor of a standard parabolic subgroup of $G$. Then $W(M)= \{1 , w_0 \}$, where $w_0$ is a Weyl group element that can be lifted to $G(\QQ)$ so that $w_0Nw_0^{-1} = N^-$ is the opposite of $N$. Recall that $P$ is said to be {\it self-associate} if $w_0 M w_0^{-1} = M$. On can prove (see \cite{Shahidi}) that $P$ as above is self-associate, unless $t$ is odd and $d=t$. Since we will only be concerned with low rank parabolic subgroups, we may  always assume that $P$ is self-associate for the sake of simplicity.

\subsection{Eisenstein series associated to non-cuspidal representations} The paper \cite{Mo08} deals with Eisenstein series where the hypothesis of the cuspidality of $\sigma\in \cA_2(M)$ is removed.  
There the following theorem is proved. 

Let $\pi=\otimes_v \pi_v$ be an irreducible sub-representation of $G(\AA)$ in $L^2 (G(\QQ) \backslash G(\AA))$ and let $\Psi$ be the global Arthur parameter of $\pi$.

\begin{theorem} \label{M}
Suppose that $\pi_\infty$ is a unitary cohomological representation. Then: either $\pi$ is cuspidal or $\Psi$ contains a factor $\rho \boxtimes R_a$ for some $a>1$ and $\rho \in \mathcal{A}_c (\GL (d))$ and there exists $\pi ' \in \mathcal{A}_2 (\SO (V_d ))$ such that 
\begin{enumerate}
\item the global Arthur parameter of $\pi'$ is obtained from that of $\pi$ by replacing the factor $\rho \boxtimes R_{a}$ by $\rho \boxtimes R_{a-2}$,
\item the Eisenstein series $E(f_s, g)$ associated to $\rho \otimes \sigma$ have at most a simple pole at $s_0=\frac12 (a -1)$, and
\item the representation $\pi$ is realized in the space spanned by the residues at $s_0= \frac12 (a-1)$ of the Eisenstein series associated to $\sigma = \rho \otimes \pi'$.
\end{enumerate}
\end{theorem}
We shall abbreviate the last item by:
\begin{equation}\label{inclusion}
\pi\hookrightarrow \left((s-(a-1)/2) E(s,\rho \times \pi ') \right)_{s=(a-1)/2}.
\end{equation}

Here we precise the preceding theorem assuming strong conditions on $\pi_{\infty}$.

\begin{proposition}\label{noncuspidalrep} 
Suppose that $\pi_{\infty}$ is isomorphic to the cohomological representation of $\SO(p,q)$ associated to the Levi subgroup $L=\SO (p-2r , q) \times \mathrm{U} (1)^r$. Then in Theorem \ref{M}, if $\pi$ is not cuspidal, we have:
\begin{enumerate}
\item the representation $\rho$ is a quadratic character (so that $d=1$) and $a=m-2r-1$, and
\item the inclusion \eqref{inclusion} is an equality.
\end{enumerate}
\end{proposition}

\begin{remark} \label{R:p}
1. Since $a\geq 2$ recall from \S \ref{S35} that being a character uniquely determines $\rho$.  

2. Since $\pi$ is irreducible, if \eqref{inclusion} is an equality then the space spanned by the residues at $s_0= \frac12 (a-1)$ of the Eisenstein series associated to $\sigma = \rho \otimes \pi'$ is irreducible. 
\end{remark}

\begin{proof}[Proof of Proposition \ref{noncuspidalrep}(1)]
Let $\pi$ be as in Theorem \ref{M} and not cuspidal. Then the Arthur parameter $\Psi$ associated to $\pi$ contains a factor $\rho \boxtimes R_a$ for some $a>1$ and $\rho \in \mathcal{A}_c (\GL (d))$ and there exists $\pi ' \in \mathcal{A}_2 (\SO (V_d ))$ such that \eqref{inclusion} holds. 

Denote by $\rho_\infty$ the local Archimedean component of $\rho$. To prove Proposition \ref{noncuspidalrep}(1),  we shall prove that $\rho_\infty$ is a character. 

Because $\pi_\infty$ is cohomological, the infinitesimal character of $\rho_\infty$ is integral and this implies that $\rho_\infty$ is an induced representation 
\begin{equation} \label{eq:t}
\mathrm{ind} (\delta_1 , \ldots , \delta_t) \quad \mbox{ or } \quad \mathrm{ind} (\delta_1 , \ldots , \delta_t , \varepsilon)
\end{equation}
(normalized induction from the standard parabolic subgroup of type $(2, \ldots , 2)$ or $(2, \ldots , 2 , 1)$ depending on the parity of $d$), where the $\delta_j$'s are discrete series of $\GL (2 , \RR)$ and $\varepsilon$ is a quadratic character of $\RR^\times$.\footnote{Note that only one quadratic character can occur because of the regularity of the infinitesimal character.} 
Now by localizing \eqref{inclusion} we conclude (see \cite{Mo08} for more details) that 
\begin{equation}\label{localinclusion}
\pi_{\infty}  \hookrightarrow \rmI \left(-\frac{a-1}{2}, \rho_{\infty} \otimes \pi_{\infty} ' \right),
\end{equation}
where $\pi'_\infty$ is the local component of $\pi'$ at infinity. Since $\pi'$ is square integrable, the representation $\pi'_\infty$ is unitary. Moreover,  after Salamanca-Riba (see \cite[Theorem 1.8]{salamanca}),  we also know that  the representation $\pi'_\infty$ is cohomological and    has a regular, integral infinitesimal character.  Write $\pi'_\infty$ as the Langlands submodule of a standard module $\rmI'$. 

We claim that the induced representation $\rmI \left(-\frac{a-1}{2}, \rho_{\infty} \otimes \rmI ' \right)$ contains a unique irreducible sub-module, which is the Langlands sub-module of this induced representation. 
Write $\rmI'$ as induced of twists of discrete series and character, see e.g. \cite[Appendix, \S 16.12]{BMM11}. In this way $\rmI \left(-\frac{a-1}{2}, \rho_{\infty} \otimes \rmI ' \right)$ is a standard module. If it is in Langlands' position,  then the claim immediately follows (see e.g. \cite[Proposition 2.6]{BW00}). To prove the claim in general,  we will reduce to this case. To do that we look at the exponents of $\rmI'$ (see again \cite[Appendix, \S 16.12]{BMM11} for more about exponents). Fix $\delta'$ a representation that occurs in the induction. Then $\delta'$ is  either a discrete series of parameter $\ell'$ or a quadratic character corresponding to $\ell'=0$.

Now let $x$ the biggest possible exponent for such a representation. Then the infinitesimal character of $\pi'$ ``contains'' the segment $\ell'+ [x,-x]$ and is symmetric around $0$. Denote by $\ell$ the parameter of one of the $\delta_j$ above. To put the induced representation in the positive Weyl chamber (i.e. in Langlands' position) we have to exchange $\delta_j \vert \mathrm{det} \vert^{(a-1)/2} $ and $\delta' \vert \mathrm{det} \vert^{x}$ if $(a-1)/2 < x$. This is done using the following intertwining operator 
(between two $\GL (4 , \RR)$-representations):
$$\mathrm{ind} (\delta_j \vert \mathrm{det} \vert^{s_1} \otimes \delta' \vert \mathrm{det} \vert^{s_2}) \mapsto \mathrm{ind} (\delta' \vert \mathrm{det} \vert^{s_2} \otimes \delta_j \vert \mathrm{det} \vert^{s_1} )$$
(normalized induction from the standard parabolic of type $(2,2)$.) To prove the claim it suffices to show that the right-hand side induced representation (which is in Langlands' position) is irreducible. We distinguish between the two following cases:

1. Let us first assume that $\ell'\neq 0$ (i.e. that $\delta'$ is a discrete series and not a character). Then it follows from \cite[Lemme I.7]{MW} (and in fact, in this case it is due to Speh) that the induced representation
$$\mathrm{ind} (\delta' \vert \mathrm{det} \vert^{x} \otimes \delta_j \vert \mathrm{det} \vert^{(a-1)/2} )$$
can be reducible only when we have both $\ell'+(a-1)/2 < \ell +x$ and $-\ell'+(a-1)/2< -\ell +x$. The second inequality is equivalent to 
$\ell + (a-1)/2 < \ell' +x$. But, by the regularity of the infinitesimal character of $\pi_\infty$, this implies $\ell+(a-1)/2< \ell' -x$. From which we conclude that 
$$x+\ell < \ell' - (a-1)/2 < \ell' + (a-1)/2.$$
This proves that our parameters avoid the bad cases (of possible reducibility). 

2. Let us now assume that $\ell'=0$. In that case, the infinitesimal character of $\pi'_\infty$ contains the segment $[x,-x]$ which is symmetric around $0$, so we certainly have $\ell+(a-1)/2> x$ and irreducibility follows. 

This concludes the proof of the claim.

Proposition \ref{noncuspidalrep}(1) follows: since $\pi_{\infty}$ is not a Langlands sub-module of a standard representation containing a non trivial twist of a discrete series, we have $t=0$ in \eqref{eq:t} and $\rho$ is a character and $d=1$. Finally, computing the L-parameter of $\pi_\infty$ and using \eqref{exponent} we conclude that $(a-1)/2 = m/2-r-1$, i.e. $a=m-2r-1$. 

\end{proof}

\begin{remark} \label{R:P(1)}
It follows from the proof above, and inspection of the Langlands parameter of $\pi_\infty$, that the representation $\pi_\infty '$ is isomorphic to the cohomological representation of $\SO(p-1,q-1)=\SO (V_1)$ associated to the Levi subgroup $L=\SO (p-1-2r , q-1) \times \mathrm{U} (1)^r$.
\end{remark}

Before entering the proof of Proposition \ref{noncuspidalrep}(2) we shall recall some basic facts on the relation between residues of Eisenstein series and residues of intertwining operators.

\subsection{Use of intertwining operators} If $P$ is self-associate and $\sigma$ is cuspidal then Langlands \cite{L76}, see also \cite[Sect. IV.1]{MW95}, has shown that the poles of $E(f_s , g)$ coincide with the poles of its constant term
$$E_P (f_s , g) = \int_{N(\QQ) \backslash N(\AA)} E(f_s , ng) dn$$
along the parabolic subgroup $P$. On the other hand, by \cite[Sect. II.1.7]{MW95}, the constant term equals 
\begin{equation}
E_P(f_s , g)=f_s(g)+M(s, w_0, \sigma )f_s(g),
\end{equation}
where 
$$M(s, w_0 , \sigma ) : \rmI (s, \sigma ) \to \rmI (-s , \sigma )$$
given --- whenever the integral below is convergent --- by 
\begin{equation} 
(M(s, w_0 , \sigma )f) (g) = \int_{N(\QQ) \cap N^- (\QQ) \backslash N(\AA)} f(w_0^{-1} n g) dn \quad (f \in \rmI (s, \sigma))
\end{equation}
is the standard intertwining operator defined in \cite[Sect. II.1.6]{MW95}.
Note that we have identified $\sigma$ with its conjugate by $w_0$ and furthermore note that $w_0 (s) =-s$. Away from its poles $M(s, w_0 , \sigma)$ intertwines $\rmI (s, \sigma )$ and $\rmI (-s , \sigma)$. 

In general,  if $\sigma$ is non-cuspidal (but discrete),  an Eisenstein series associated to $\sigma$ is holomorphic at $s=s_0$ if and only if all its cuspidal constant terms are holomorphic, see e.g. \cite[Sect. I.4.10]{MW95}. If $\sigma=\rho\otimes \pi'$, as in Theorem \ref{M}, the cuspidal constants terms are obtained as linear combinations of intertwining operators defined on representations induced from $\rho | \cdot |^s \otimes \pi'_Q$ where $\pi'_Q$ is a cuspidal constant term along the unipotent radical of some standard parabolic subgroup $Q$. 

In order to understand the singularities of the Eisenstein series, one is therefore reduced to study the singularities of the standard intertwining operators. 

\subsection{Normalization of intertwining operators} 
Intertwining operators are meromorphic and can be decomposed as Eulerian products using local intertwining operators (that are also meromorphic); see \cite[\S II.1.9]{MW95}.
Consider a square integrable representation $\sigma=\rho \otimes\pi ' \in\cA_2(M)$, as in Theorem \ref{M}, with  $\rho \in \cA_c(\GL(d))$ and $\pi ' \in \cA_2(\SO(V_d))$.
If $v=p$ is a finite place where everything is unramified the value of a local (at $p$) intertwining operator is --- up to an invertible holomorphic function --- an explicit expression of local L-functions:
$$\frac{L_p ( 2s, \rho_p \times (\pi_{p} ')^{\rm GL})}{L_p (2s+1, \rho_p \times (\pi_p ')^{\rm GL} )}\cdot \frac{L_p (2s,r_G, \rho_p )}{L_p ( 2s+1, r_G, \rho_p )},$$
where $(\pi_p ')^{\rm GL}$ is the local functorial lift of $\pi_p$ to a representation of $\GL_{N_1} (\QQ_p)$ (with $N_1 = N-2d$), see e.g. \cite{Shahidi}. 

Note that if $\Psi '$ is the global Arthur parameter of $\pi '$, it follows from Proposition \ref{prop:A2} (and the fact that being unramified we have $p \notin S$) that $(\pi_p ')^{\rm GL}$ is associated to the $L$-parameter $\phi_{\Psi_p '}$. This suggests to normalize the global intertwining operators using a product of $L$-functions (and there inverses) analogue as the one above but using the global Arthur parameter $\Psi'$: let $\Pi_{\Psi '}$ be the representation of $\GL(N_1)$ defined in \eqref{transfer}. In \cite{Mo08}, the rank one global standard intertwining operator $M(s,w_0, \rho\otimes \pi ')$ are normalized in terms of the Arthur parameter $\Psi '$,
\begin{equation}
N_{\Psi '} (s,w_0, \rho\otimes\pi'):=r(s,\Psi ' )^{-1} M(s, w_0, \rho\otimes \pi ' ),
\end{equation}
where the normalizing factor 
\begin{equation}\label{normalizingfactor}
r(s,\Psi ')= \frac{L(2s, \rho \times \Pi_{\Psi '})}{L(2s+1, \rho \times \Pi_{\Psi '} )}\cdot \frac{L(2s,r_G, \rho )}{L( 2s+1, r_G, \rho)}
\end{equation}
is the quotient of $L$-functions defined in \S \ref{S27}. Write $N_{\Psi '}(s, w_0 , \rho \times \pi ')=\otimes_v N_{\Psi_v ' }(s,w_0, \rho_v\times \pi_v ')$. Then if $v=p$ is unramified, the local factor of $N_{\Psi '} (s,w_0, \rho\times\pi ')$ at $p$ is the identity on spherical vectors and at any place $v$ 
$$N_{\Psi_v ' }(s,w_0, \rho_v\otimes \pi_v ' )=r(s,\Psi_v ' )^{-1} M(s,w_0, \rho_v\otimes \pi_v ' )$$ 
intertwines the induced representations $\rmI(s, \rho_v\otimes \pi_v ' )$ and $\rmI(-s,\rho_v\otimes\pi_v ' )$.  
If one can show that these normalized intertwining operators are holomorphic, one gets that the poles of the global intertwining operator $M(s, w_0 , \eta \otimes \pi ' )$ are precisely the poles of the global
normalizing factor $r(s,\Psi ' )$. The latter being expressed in terms of standard $L$-functions, we can easily determine its poles.

\subsection{Proof of Proposition \ref{noncuspidalrep}(2)}
Let $\pi$ be as in Theorem \ref{M} and not cuspidal. Since Proposition \ref{noncuspidalrep}(1) is proved, the Arthur parameter $\Psi$ associated to $\pi$ contains a factor $\eta \boxtimes R_a$ for some $a>1$ and $\eta$ a quadratic character and there exists $\pi ' \in \mathcal{A}_2 (\SO (V_1 ))$ such that \eqref{inclusion} holds. Moreover: the global Arthur parameter of $\pi'$ is obtained from that of $\pi$ by replacing the factor $\eta \boxtimes R_{a}$ by $\eta \boxtimes R_{a-2}$.

Recall the following notation introduced in \cite{Mo08}: if $\widetilde{\pi}$ is any automorphic representation (not necessarily irreducible), $P\subset P_d$ is a standard maximal parabolic subgroup and $\rho \in \mathcal{A}_c (\GL (d))$, we denote by $\widetilde{\pi}_P [\rho ]$ the projection on the $\rho$-cuspidal part of the constant term of $\widetilde{\pi}$ along $P$. 

Now let $\widetilde{\pi}$ be the (automorphic) representation in the space spanned by the residues at $s_0= \frac12 (a-1)$ of the Eisenstein series associated to $\sigma = \eta \otimes \pi'$.

\begin{lemma} \label{L:2}
Suppose that $\pi_{\infty}$ is isomorphic to the cohomological representation of $\SO(p,q)$ associated to the Levi subgroup $L=\SO (p-2r , q) \times \mathrm{U} (1)^r$. Then $\widetilde{\pi}_P [\rho] = \{ 0 \}$ unless 
$P=P_1$ (with Levi $\GL(1) \times \SO (V_1)$) and $\rho = \eta$. Furthermore:  in this latter case $\widetilde{\pi}_P [\rho]$ is precisely the image of the induced representation 
$\rmI (s_0 , \eta \otimes \pi ' )$ by the intertwining operator $\left((s-s_0) M(s, w_0, \eta \otimes \pi ' )\right)_{s=s_0}$.\footnote{Note that this operator is well defined since, as a pole of the Eisenstein series associated to $\eta \otimes \pi '$, $s_0$ is at most simple (see Theorem \ref{M}).} 
\end{lemma}
\begin{proof}[Proof of Lemma \ref{L:2}]We first remark that  if there is only  $\widetilde{\pi}_{P_1} [\eta]\neq 0$, then the last statement  follows from \cite[\S II.1.7]{MW95}.  So it suffices to show that only  $\widetilde{\pi}_{P_1} [\eta]$ is nonzero.  According to  Remark \ref{R:P(1)},   we know that $\pi_\infty '$ is isomorphic to the cohomological representation of $\SO(p-1,q-1)=\SO (V_1)$ associated to the Levi subgroup $L=\SO (p-1-2r , q-1) \times \mathrm{U} (1)^r$. By induction (on $p+q$),  we may therefore assume that Lemma \ref{L:2} holds for $\pi '$. (Note that the induction start with $\dim V = 2$ or $3$ where the Lemma is obvious.)

Now if $\widetilde{\pi}_P [\rho] \neq \{ 0 \}$ for some $P \neq P_1$ (and some $\rho$), then $\pi '$ is not cuspidal. (Recall that Langlands has proved that if $\sigma=\eta \otimes \pi'$ is cuspidal the associated Eisenstein series has only one non-trivial constant term, along $P_1$.) And since constant terms of Eisenstein series are sums of Eisenstein series of constant terms (see \cite[\S II.1.7]{MW95}), we have $\pi_{P'} ' [\rho ' ] \neq \{ 0 \}$ for some standard parabolic subgroup $P' \subset \SO (V_1)$ with Levi isomorphic to $\GL(d) \times \SO (V_{d+1})$ and some $\rho ' \in \mathcal{A}_c (\GL (d))$. Now since Lemma \ref{L:2} holds for $\pi '$, we conclude that $d=1$,  $\rho ' =\eta$ and  $\pi'$ must be a quotient of an induced representation $\rmI (s_0 , \eta \otimes \pi '' )$ for some square integrable automorphic representation $\pi ''$ on $\SO(V_{2})$ with the {\it same} $s_0$. But this is not possible because of the form of the Arthur parameter of $\pi '$. Similarly,  one shows that $\widetilde{\pi}_{P_1} [\rho]= 0$ unless $\rho=\eta$.
\end{proof}

Let us now prove that Lemma \ref{L:2} implies Proposition \ref{noncuspidalrep}(2). The representation $\widetilde{\pi}$ injects into the sum of its constant terms along standard maximal parabolic subgroups. But this sum reduces to the image of the induced representation $\rmI (s_0 , \eta \otimes \pi ' )$ by the intertwining operator $\left((s-s_0) M(s, w_0, \eta \otimes \pi ' )\right)_{s=s_0}$. We claim that the image of this  intertwining operator is irreducible.  This implies that $\widetilde{\pi}$ is irreducible and the proposition follows.  

To prove the claim,  let us apply $\left((s-s_0) M(s , w_0 , \eta \otimes \pi ' )\right)_{s=s_0}$ to some function $f$ which is the characteristic function of a hyperspecial maximal compact subgroup outside a finite set of places. Let $S$ be a finite set of places that contains all these bad places and all the places of ramification. Then for any
$v \notin S$,  the normalized local intertwining operator  
\begin{equation} \label{eq:Nv}
N_{\Psi_v '} (s , w_0 , \eta_v \otimes \pi_v ' ) = r(s,\Psi_v ' )^{-1} M (s,w_0, \eta_v \otimes \pi_v ' )
\end{equation}
is the identity. It is proved in \cite[\S 5.3.1]{Mo11} that for any finite $v \in S$ the image of the local intertwining operator \eqref{eq:Nv} is either $0$ or irreducible. We shall now prove that this is still true if $v=\infty$. First of all we 
remark that the standard intertwining operator $M (s , w_0 , \eta_\infty \otimes \pi_\infty ' )$ is holomorphic at $s=s_0$ with irreducible image. Indeed we have:
$$M(s , w_0 , \eta_\infty \otimes \pi_\infty ' ) : \rmI (s , \eta_\infty \otimes \pi_\infty ' ) \to \rmI (-s , \eta_\infty \otimes \pi_\infty ' ) .$$  
But the exponents of $\pi_\infty '$ are strictly less that $s_0$ (see the proof of Proposition \ref{noncuspidalrep}(1)).  So the induced modules are in Langlands position for $s=s_0$ and it follows from \cite[Proposition 2.6]{BW00} that $M (s , w_0 , \eta_\infty \otimes \pi_\infty ' )$ is holomorphic in $s=s_0$ with irreducible image. It remains to prove that the local normalization factor does not add a pole, i.e. that $r(s,\Psi_\infty ' )$ does not vanish in $s=s_0$. But this follows from the explicit formula for $r(s,\Psi_\infty ' )$, see \cite{MW} for details. 

Since the global normalization factor has a simple pole at $s=s_0$,  we know that the image the intertwining operator $\left((s-s_0) M(s , w_0 , \eta \otimes \pi ' )\right)_{s=s_0}$ is either $0$ or irreducible. As it contains the space of $\pi$ which  is not $0$,  this concludes the proof of our claim and therefore of Proposition \ref{noncuspidalrep}(2). \qed

\section{Theta correspondence for orthogonal groups} The theory of Howe's theta correspondence between reductive dual pairs has been applied to construct many automorphic representations with non-zero cohomology at infinity (cf.~\cite{Li}).  Conversely, the results of \cite{BMM11} show that small degree cohomological \emph{cuspidal} representations of $\SO(V)$ come from the theta correspondence. In this section, we generalize the latter to non-cuspidal automorphic representations. 

\subsection{Theta correspondence}
In this section, we will consider the reductive dual pair $(\mathrm{O}(V), \Sp(W))$, where $W$ is a symplectic space over $\QQ$ of dimension $2r$ and $\Sp(W)$  is the associated symplectic group. 
Let $\widetilde{\Sp}(\cdot)$ be the metaplectic double cover of the symplectic group $\Sp(\cdot)$. Then we denote by $\Mp(W)$  the symplectic group $\Sp(W)$ if $m$ is odd  and  the metaplectic group $\widetilde{\Sp}(W)$ if $m$ is even.  There is a natural morphism 
\begin{equation}\label{reductivedual}
\mathrm{O}(V)\times \Mp(W)\rightarrow \widetilde{\Sp}(V\otimes W).
\end{equation}
We denote by $\Mp_{2r}(\AA)$, $\mathrm{O}_m(\AA)$ and $\widetilde{\Sp}_{2mr}(\AA)$ the adelic points of the groups $\Mp(W)$, $\mathrm{O}(V)$ and $\widetilde{\Sp}(V\otimes W)$ respectively. 

Fix a nontrivial additive character $\psi$ of $\AA/\QQ$,  let $\omega_\psi$ be the (automorphic) Weil representation of  $\Mp_{2mr}(\AA)$ realized in the Schr$\ddot{\hbox{o}}$dinger model $\cS(V^{r}(\AA))$,  where $\cS(V^{r}(\AA))$ is the space of Schwartz-Bruhat  functions on $V^n(\AA)$. By restricting the Weil representation $\omega_\psi$ to $\mathrm{O}(V)\times \Mp(W)$,  one can  define the theta function on $\mathrm{O}(V)\times \Mp(W)$:
\begin{equation}\label{thetafunction}
\theta_{\psi,\phi}(g,g')=\sum\limits_{\xi \in V(\QQ)^r} \omega_\psi(g,g') (\phi)(\xi),
\end{equation}
for each $\phi\in \cS(V^{r}(\AA))$. Given a cuspidal representation $(\tau,H_\tau)$ of $\Mp_{2r}(\AA)$ and $f\in H_\tau$, the integral 
\begin{equation}\label{theta}
\theta_{\psi,\phi}^{f}(g)=\int\limits_{\Mp(\QQ)\backslash\Mp(\AA)}\theta_{\psi,\phi}(g,g') f(g') dg',
\end{equation}
defines  an automorphic function on $\mathrm{O}_m(\AA)$, called the {\it global theta lift} of $f$.  
Let $\theta_{\psi,V} (\tau)$ be the space of all the theta liftings $\theta_{\psi,\phi}^f$ as  $f$ and $\phi$ vary. This is called the global $\psi$-theta lifting of $\tau$ to $\mathrm{O}_m(\AA)$. 

\begin{remark} When the space $\theta_{\psi,V} (\tau)$ contains a nonzero  cuspidal automorphic function,  it is known that  $\tilde{\pi}\cong\theta_{\psi,V}(\tau)$ is irreducible and cuspidal (cf.~\cite{MoeglinJLT1,JS07}). Moreover, it will be the first occurrence of the global $\psi$-theta lifting of $\tau$ in the Witt tower of quadratic spaces (cf.~\cite{Ra84}). 
\end{remark}

\subsection{} For the  special orthogonal group $\SO(V)$, we say that  $\pi\in \cA_2(\SO(V))$  is in the image of $\psi$-cuspidal theta correspondence from a smaller symplectic group if a lift $\tilde{\pi}$ of $\pi$ to $\mathrm{O}(V)$ is  contained in a global $\psi$-theta lifting from symplectic groups,  i.e. there exists $\tau\in\cA_c(\Mp(W))$ such that  $\tilde{\pi}\hookrightarrow \theta_{\psi,V}(\tau)$. 
 The restriction of $\tilde{\pi}$ to $\SO(V)$ is isomorphic to $\pi$ and its lift $\tilde{\pi}$ should be uniquely determined.  
 Then one of the key results in \cite{BMM11}  states as 

\begin{proposition}\label{P:BMM}\cite[Theorem 4.2]{BMM11} Let $\pi$ be a cuspidal automorphic representation of $G(\AA)$ with regular infinitesimal character. If $\pi$ is highly non-tempered,   then there exists a cuspidal representation $\tau$ of $\Mp_{2r}(\AA)$ such that 
$\pi$ (up to a twist by a quadratic character)  is in the image of $\psi$-cuspidal theta correspondence from a smaller symplectic group.
\end{proposition}

\subsection{} We also recall the local theta correspondence which will be used later. For temporary notations, we let $G=\mathrm{O}(V)(\QQ_v)$ and $G'=\Mp(W)(\QQ_v)$. Similarly as the global case, the local theta correspondence is obtained from the restriction of the local Weil representation on $G\times G'$.  By abuse of notations, we denote by $\omega_{\psi,v}$ the pullback of the Weil representation to $G\times G'$. Then for  an irreducible admissible representation $\tau$  of $G'$, we  can define 
$$S(\tau)=\omega_{\psi,v}/\bigcap\limits_{\phi\in\Hom_{G}(\omega_{\psi,v},\tau) }\ker \phi$$
to be the maximal  quotient of $\omega_{\psi,v}$ on which $G$ acts as a multiple of $\tau$. By \cite[Lemma III.4]{MVW} if $v$ is finite there is a smooth representation $\Theta(\tau)$ of $G$ such that 
$$S(\tau)\simeq\tau\otimes\Theta(\tau)$$
and $\Theta(\tau)$ is unique up to isomorphism. The Howe duality principle (cf.~\cite{Ho89,Wa90,GT}) asserts further that if $\Theta(\tau)\neq 0$, then it has a unique irreducible quotient $\theta_{\psi,v}(\tau)$ and  two representations $\tau_1, \tau_2$ are isomorphic if $\theta_{\psi,v}(\tau_1)\cong \theta_{\psi,v}(\tau_2)$. If $v=\infty$, Howe \cite{Ho89} proves that if $\pi$ and $\pi'$ are two irreducible $(\mathfrak{g} , K)$-modules for $G$ and $G'$ then, in the category of $(\mathfrak{g} , K)$-modules, we have:
$$\dim \mathrm{Hom}_{G\times G'} (\omega_{\psi, v} , \pi \otimes \pi ') \leq 1.$$
This will be enough for our purpose: there is at most one irreducible representation $\theta_{\psi,v}(\tau)$ such that $\mathrm{Hom}_{G\times G'} (\omega_{\psi, v} , \tau \otimes   \theta_{\psi,v}(\tau) ) \neq \{ 0 \}$. The irreducible representation $\theta_{\psi,v}(\tau)$ is called the {\it local $\psi$-theta lifting} of $\tau$.  
As indicated by Rallis in \cite{Ra84}, there is a natural relation between global and local theta liftings:

\begin{proposition}  
For $\tau=\otimes_v\tau_v$ of $\Mp_{2r}(\AA)$, the local theta lifting $\theta_{\psi,v}(\tau_v)$ occurs as the irreducible constituent of the local component of the global theta lifting  $\theta_{\psi,V}(\tau)$.   In particular,  if $\theta_{\psi,V}(\tau)$  is cuspidal, then $\theta_{\psi,V}(\tau)\cong \otimes_v \theta_{\psi,v}(\tau_v)$. 
\end{proposition}

\begin{remark}
At the Archimedean place $v=\infty$, Li \cite{LiDuke} has shown that the cohomological $(\frg,K)$-modules $A_\frq $ of Proposition \ref{cohorep} occur as $(\frg,K)$-modules of representations obtained as local theta lifts of 
holomorphic unitary discrete series representation of $G'(\RR)=\Mp_{2r}(\RR)$. 
\end{remark}

\subsection{Constant terms of theta lifts} Next, we recall the constant term of global theta lifts computed  by Rallis.  With the notation as before,  let $\tilde{P}_i$ be the standard maximal parabolic group of $\mathrm{O}(V)$ with the  Levi decomposition $\tilde{P}_i=\tilde{M}_i\tilde{N}_i$ and 
$$\tilde{M}_i=\GL_i\times \mathrm{O}(V_i), \quad i \leq t.$$
For a global theta lift $\theta^f_{\psi,\phi}$ defined in \eqref{theta},  Rallis \cite{Ra84} has  shown that the constant term 
\begin{equation}\label{constanterm}
\theta_{\psi,\phi}^f(g)_{(\tilde{P}_i)}=\int_{\tilde{N}_i(\QQ)\backslash \tilde{N}_i(\AA)}\theta_{\psi,\phi}^f(ng)dn
\end{equation}
is the lift associated to a theta series in a smaller number of variables. More precisely, using the decomposition $V=U_{i}+V_{i}+U'_{i}$, one can write a vector $v\in V^r$ as 
$$v=\left[\begin{array}{c}u_i \\\hline v_i \\\hline u'_i\end{array}\right], \quad u_i\in U_{i}^r, \ v_i\in V_{i}^r, \ u_i ' \in (U_{i}')^r . $$ 
Then we have 
\begin{lemma}\label{constanttermofthetalift} Let $(\tau,H_\tau)$ be an irreducible cuspidal representation of $\Mp(W)$. 
For a cusp form $f\in H_\tau$,   the constant term of $\theta^f_{\psi,V}$ along the unipotent radical $N_i$ is
\begin{equation}\label{constanterm1}
\begin{aligned}
\theta^{f}_{\psi, \phi}(g)_{(\tilde{P}_i)}
=\int_{\Mp_{2n}(\QQ)\backslash \Mp_{2n}(\AA)}\theta'(g',g) f(g')dg' 
\end{aligned}
\end{equation}
where 
$$\theta'(g',g)=\sum\limits_{\xi_i\in V_{i}^r(\QQ)} \int\limits_{u_i\in U^r_i(\AA)} \omega_\psi(g',g) (\phi)\left[\begin{array}{c}u_i \\\hline \xi_i \\\hline 0 \end{array}\right] du_i.$$ 
\end{lemma}
\begin{proof} The computation  \eqref{constanterm1} can be found in \cite{Ra84} Theorem I.1.1 (2) when $m$ is even. The proof is similar when $m$ is odd.
\end{proof}

From Lemma \ref{constanttermofthetalift}, we see that  the constant term of $\theta_{\psi,V}(\tau)$ along $P_i$ lies exactly in the theta lifting of $\tau$ to the smaller orthogonal group $\mathrm{O}(V_i)$ where $\GL_i$ acts by a character. So the only nonzero cuspidal constant term of $\theta_{\psi,V}(\tau)$ is along the parabolic subgroup contained in some $P_{s}$  ($s \leq d$) and stabilizing the flag consisting of $s$ isotropic spaces. Moreover: this constant term is isomorphic to 
\begin{equation}
\eta |\cdot|^{-m/2+r+1}_\AA\otimes \cdots \otimes \eta |\cdot|^{-m/2+r+s}_\AA\otimes \theta_{\psi, V_s}(\tau),
\end{equation}  
where the irreducible representation $\theta_{\psi,V_s}(\tau)$ is the first occurrence (and cuspidal by \cite{Ra84}) of the theta correspondence of $\tau$, and $\eta$ is a character that only depends on the choices made for the Weil representation (see \cite[Remark 2.3]{Kudla86} where Kudla computes the Jacquet modules of local theta lift, see also \cite[p. 69]{MVW}).  
By the square integrability criterion in  \cite[I.4.11]{MW95}, each automorphic representation  $\theta_{\psi,V_i}(\tau)$ is square integrable for $i\leq s$. Specializing Lemma \ref{constanttermofthetalift} we can compute its constant term:

\begin{corollary}\label{constanterm3}
The constant term of $\theta_{\psi,V_i}(\tau)$ along the unipotent radical of a parabolic stabilizing an isotropic line of $V_i$ is equal to 
$$\eta | \cdot |^{-m/2+r+i+1}_\AA\otimes \theta_{\psi,V_{i+1}}(\tau).$$
\end{corollary}

\subsection{A surjectivity  theorem}
The following result shows that, under some conditions, representations with cohomology come from the theta correspondence.

\begin{theorem}\label{surjoftheta}
 Let $\pi\in \cA_2(\SO(V))$ be a square integrable automorphic representation of $\SO(V)$. Suppose that the local Archimedean component of $\pi$ is isomorphic to the cohomological representation of $\SO(p,q)$ associated to the Levi subgroup $L=\SO (p-2r , q) \times \mathrm{U} (1)^r$. Then, if $3r<m-1$, there exists a cuspidal representation $\tau$ of $\Mp_{2r}(\AA)$ such that $\pi$ (up to a twist by a quadratic character) is in the image of theta lift of $\tau$.
\end{theorem}
\begin{proof} By Proposition \ref{P:BMM}, it suffices to consider the case when $\pi$ is non-cuspidal. It then follows from Proposition \ref{noncuspidalrep} that the Arthur parameter of $\pi$ contains a factor $\eta \boxtimes R_a$, where $\eta$ is a quadratic character and $a=m-2r-1>1$. 

Now note that the Arthur parameter of $\pi_{1}= \pi'$ in Proposition \ref{noncuspidalrep} is obtained by replacing the factor $\eta \boxtimes R_a$ in the parameter of $\pi$ by the factor $\eta \boxtimes R_{a-2}$. So if $\pi_{1}$ is non-cuspidal and $a-2>1$,  it will again be the residue of Eisenstein series for some $\pi_2\in \cA_2 (\SO(V_2))$ by applying Proposition 
\ref{noncuspidalrep} one more time.

If $\pi_1$ is cuspidal, we shall put $t=1$ and otherwise inductively define $\pi_{2}, \ldots, \pi_{t}$ until we arrive at a cuspidal representation $\pi_t$. Note that the processus stops at finite $t$ less or equal to the Witt index of $V$ and even $t\leq \min (p-2r,q)$. We furthermore remark that $m/2-r-i+1=(p-2r+q-2i)/2+1>0$. We therefore obtain a sequence of irreducible representations $\pi_i\in\cA_2(\SO(V_i))$, $i=0,1,2\ldots t$, such that  
\begin{enumerate}
\item $\pi_0=\pi$, we have
$$\pi_{i-1}=\left( (s_i-m/2+r+i)E(s_i , \eta\times \pi_{i} ) \right)_{s_i=m/2-r-i},$$ 
for each $i\leq t$, and  $\pi_t\in\cA_c(\SO(V_t))$ is cuspidal;
\item the Arthur parameter of $\pi_{i+1}$ is the same as the Arthur parameter of $\pi_{i}$ except that the factor $\eta\boxtimes R_{m-2r-2i-1}$ is replaced by $\eta\boxtimes R_{m-2r-2i-3}$;
\item at the Archimedean place $v=\infty$, the underlying $(\frg,K_{\infty})$-module $(\pi_i)^\infty_{\infty}\cong A_{\frq_i}$, where $A_{\frq_i}$ is  associated to the Levi subgroup $\SO(p-2r-i,q-i)\times \mathrm{U}(1)^r\subseteq \SO(p-i,q-i)$.
\end{enumerate}
It follows from Theorem \ref{nontemper} that $\pi_{t}$ is the image of a cuspidal representation $\tau$ of $\mathrm{Sp}(2r)$ using the $\theta$-lifting: if $t=0$, we use the hypothesis $r<(m-1)/3$ and the assertion follows from Proposition \ref{P:BMM}. If $t>0$, the wanted property follows from \cite[Theorem on p. 203]{MoeglinJLT1} if $m$ is even, and  \cite[Theorem 5.1]{GJS} if $m$ is odd.\footnote{It is not explicitely stated there that $\tau$ is cuspidal, but this follows from the Rallis theta tower property \cite{Ra84}, since $\tau$ is the first occurence.}

We now want to prove that $\pi$ itself is in the image of the theta correspondence.
Let us take $\pi_t'=\pi_t$. According to Corollary \ref{constanterm3} we can define inductively a sequence $\pi'_i$ ($0\leq i\leq t$) of square integrable automorphic representation of $\SO (V_i)$ as follows:  let $\pi_i'$ be the irreducible constituent in the $\theta$-lift of $\tau$ on $\SO(V_i)$ whose constant term along the unipotent radical of a parabolic stabilizing an isotropic line of $V_i$ contains $\eta|\cdot |_\AA^{-m/2+r+i+1}\otimes \pi_{i+1}'$. 

At infinity, the Archimedean component $(\pi_i')_{\infty}$ is in the image of local theta correspondence of $\tau_{\infty}$. The local correspondence is known in this case, see \cite{LiDuke}, and its underlying $(\frg,K)$-module $(\pi'_i)^\infty_{\infty} \cong A_{\frq_i}$. We prove by induction that $\pi'_i=\pi_i$. This is true by construction of $\pi'_t$ for $i=t$. Assume that it is true until $i$ and let us prove it for $i-1$. We know that $\pi'_{i-1}$ is not cuspidal and that it satisfies the condition Proposition \ref{noncuspidalrep}. We may therefore realize $\pi'_{i-1}$ in the space of residues of the Eisenstein series constructed from some character $\eta'$ and some irreducible representation $\pi' \in \mathcal{A}_2 (\SO (V_i))$. Computing the constant term as explained above, we obtain that $\eta' = \eta$ and $\pi' = \pi'_i$. Now using the induction equality $\pi'_i=\pi_i$ and the irreducibility of the space of residues already proved, we conclude that $\pi'_i=\pi_i$, as wanted. This proves that $\pi=\pi'_0$ is in the image of the theta correspondence.
\end{proof}

\section{Cohomology of arithmetic manifolds}

\subsection{Notations} 
Let us  take 
$$\widehat{D}=G(\RR )/(\SO(p)\times \SO(q)),$$ 
and let $D$ be  a connected component of $\widehat{D}$. Let $\widetilde{G}$ be the general spin group $\Gspin(V)$ associated to $V$. For any compact open subgroup $K\subseteq G(\AA_f)$, we set $\widetilde{K}$ to be its preimage in $\widetilde{G}(\AA_f)$.  Then  we denote by $X_K$ the double coset
$$\widetilde{G}(\QQ)\backslash (\SO(p,q)\times \widetilde{G} (\AA_f)) / (\SO(p)\times \SO(q)) \widetilde{K}.$$

Let $\widetilde{G}(\QQ)_+\subseteq \widetilde{G}(\QQ)$ be the subgroup consisting of elements with totally positive spinor norm, which can be viewed as the subgroup of  $\widetilde{G}(\QQ)$ lying in the identity component of the adjoint group of $\widetilde{G}(\RR )$. Write
$$\widetilde{G}(\AA_f)=\coprod_{j} \widetilde{G}(\QQ)_+ g_j \widetilde{K},$$
one has that the decomposition of $X_K$ into connected components is 
 $$X_K=\coprod\limits_{g_j} \Gamma_{g_j}\backslash D,$$
where  $\Gamma_{g_j}$ is the image of $\widetilde{G}(\QQ)_+\cap g_j \tK g_j^{-1}$ in $\SO_0(p,q)$. When $g_j=1$, we denote by $\Gamma_K$ the arithmetic group $\Gamma_1=K\cap G(\QQ)$ and $Y_K=\Gamma_K\backslash D$ the connected component of $X_K$.   Throughout this section, we assume that $\Gamma_1$ is torsion free. The arithmetic manifold $Y_K$ inherits a natural Riemannian metric from the Killing form on the Lie algebra of $G( \RR )$,  making it a complete manifold of finite volume.

\subsection{$L^2$-cohomology on arithmetic manifolds} 
Let $\Omega_{(2)}^i(Y_K,\CC)$ be the space of $\CC$-valued smooth square integrable $i$-forms on $Y_K$ whose exterior derivatives are still square integrable. It forms a complex $\Omega_{(2)}^\bullet(Y_K,\CC) $ under the natural exterior differential operator $$d:\Omega_{(2)}^i(Y_K,\CC)\rightarrow \Omega_{(2)}^{i+1}(Y_K,\CC).$$  The $L^2$-cohomology $H^\ast_{(2)}(Y_K,\CC)$ of $Y_K$ is defined as the cohomology of the complex $\Omega_{(2)}^\bullet(Y_K,\CC) $.
With the distribution exterior derivative $\bar{d}$, 
one can work with the full $L^2$-spaces $L_{(2)}^i(Y_K,\CC)$
instead of just smooth  forms,  i.e. $\omega\in L_{(2)}^i(Y_K,\CC)$ is a square integrable  $i$-form and $\bar{d}\omega$ remains square integrable,  then we can define the reduced $L^2$-cohomology group to be 
$$\bar{H}_{(2)}^i(Y_K,\CC)=\{\omega\in L_{2}^i(Y_K,\CC):\bar{d}\omega=0\}/\overline{\{\bar{d} L_{(2)}^{i-1}(Y_K,\CC) \}},$$
where $\overline{\{\bar{d} L_{(2)}^{i-1}(Y_K,\CC) \}}$ denotes the closure of the image ${\rm Im} \bar{d}$ in $ L_{2}^i(Y_K,\CC)$. 

By Hodge theory, the group $\bar{H}^i_{(2)}(Y_K,\CC)$ is isomorphic to the space of $L^2$-harmonic $i$-forms,  which is a finite dimensional vector space with a natural Hodge structure (cf.~\cite{BG83}). 
As $Y_K$ is complete, there is an inclusion
\begin{equation}\label{inclusionL2}
\bar{H}_{(2)}^i(Y_K,\CC)\hookrightarrow H_{(2)}^i(Y_K,\CC),
\end{equation}
and it  is an isomorphism when  $H_{(2)}^i(Y_K,\CC)$ is finite dimensional.

Let $\Omega^i(Y_K,\CC)$ be the  space of smooth $i$-forms on $Y_K$. The inclusion 
$$\Omega^i_{(2)}(Y_K,\CC)\hookrightarrow \Omega^i(Y_K,\CC)$$
induces a homomorphism  
\begin{equation}\label{L2map}
H^i_{(2)}(Y_K, \CC)\rightarrow H^i(Y_K,\CC),
\end{equation} 
between the $L^2$-cohomology group and ordinary de-Rham cohomology group. We denote by $\bar{H}^i(Y_K,\CC)$ the image of $\bar{H}^i_{(2)}(Y_K,\CC)$ in $H^i(Y_K,\CC)$. In general, the mapping \eqref{L2map} is neither injective nor surjective, but  as we will see later, \eqref{inclusionL2} and \eqref{L2map} become isomorphisms when $i$ is sufficiently small.

\subsection{Zucker's Result}
Now we  review Zucker and Borel's results on comparing $L^2$-cohomology groups with ordinary de Rham cohomology groups.   
Let $P_0$ be a minimal parabolic subgroup of $G ( \RR )$ and $\frq_0$ the associated Lie algebra with Levi decomposition 
$\frq_0=\frl_0+\fru_0.$
Let $A\subseteq P_0$ be the maximal $\QQ$-split torus and  $\fra_0$ the associated Lie algebra. We consider the Lie algebra cohomology $H^\ast(\fru,\CC)$ as a $\frl_0$-module. Then Zucker (see also \cite{Bo81}) shows that 
 
\begin{theorem}\cite[Theorem 3.20]{Zu82}
The mapping \eqref{L2map} is an isomorphism for $i\leq c_G $, where the constant
$$c_G=max\{k:\beta+\rho>0~\hbox{for all weights $\beta$ of $H^k(\fru,\CC)$}\},$$
where $\rho$ be the half sum of positive roots of $\fra_0$ in $\fru$. In particular, $c_G$ is at least the greatest number of $\{k:\beta+\rho>0~\hbox{for all weights $\beta$ of $\wedge^k\fru^\ast$}\}$ which is greater than $[\frac{m}{4}]$.
\end{theorem}

\begin{remark}
The result in \cite{Zu82} is actually much more general. Zucker has shown the existence of such a constant not only for cohomology groups with trivial coefficients, but also for cohomology groups with non-trivial coefficients. 
\end{remark}

\begin{example}
When $G(\RR)=\SO(p,1)$ and $Y_K$ is a hyperbolic manifold, the constant $c_G$ is equal to $[\frac{p}{2}]-1$ (cf.~\cite[Theorem 6.2]{Zu82}) and thus we have the isomorphisms
\begin{equation}\label{hyper-L2}
\bar{H}^i_{(2)}(Y_K,\CC)\xrightarrow{\sim}H^i_{(2)}(Y_K,\CC)\xrightarrow{\sim}H^i(Y_K,\CC)
\end{equation}
for $i\leq [\frac{p}{2}]-1$.
\end{example}

\begin{example}\label{zucker-conj}
In case $G(\RR )=\SO(p,2)$ and $Y_K$ is locally Hermitian symmetric, we can have a better bound for $i$ from Zucker's conjecture to ensure \eqref{L2map} being an isomorphism.  Remember that the quotient $Y_K$ is  a quasi-projective variety with the Baily-Borel-Satake compactification $\overline{Y}_K^{bb}$,  then Zucker's conjecture (cf.~\cite{Lo88,SS90}) asserts that  there is an isomorphism  
\begin{equation}\label{Zucker}
H^i_{(2)}(Y_K,\CC)\cong IH^i(Y_K,\CC),
\end{equation}
where $IH^i(Y_K,\CC)$ is the intersection cohomology on $\overline{Y}_K^{bb}$. Since the boundary of $\overline{Y}_K^{bb}$ has dimension at most one, we have an isomorphism 
\begin{equation}\label{L2-Zucker-Shimura}
H^i_{(2)}(Y_K,\CC) \cong IH^i(Y_K,\CC)\cong H^{i}(Y_K,\CC)=\bar{H}^{i}(Y_K,\CC),
\end{equation}
for $i<p-1$. Moreover, a result of Harris and Zucker (cf.~\cite[Theorem 5.4]{HZ01}) shows that the map \eqref{L2-Zucker-Shimura} is also a Hodge structure morphism. Therefore, $H^i(Y_K,\CC)$ has a pure Hodge structure when $i<p-1$.
\end{example}

\subsection{Spectrum decomposition of cohomology}
Let $L^2(\Gamma_K\backslash G ( \RR) )$ be the space of square integrable functions on $\Gamma_K\backslash G (\RR )$ and let $L^2(\Gamma_K\backslash G ( \RR))^\infty$ be the subspace of smooth vectors. It is a $(\frg,K_\infty )$-module and  Borel shows that
$$H^\ast_{(2)}(Y_K,\CC )\cong H^\ast (\frg,K_\infty ; L^2(\Gamma_K\backslash G(\RR) )^\infty).$$
We can exploit Langlands' spectral decomposition of $L^2(\Gamma_K\backslash G (\RR ))$ to obtain the corresponding decomposition of the cohomology group. 

Let $L^2_{\rm dis}(\Gamma_K\backslash G (\RR ))^\infty$ be the discrete spectrum of $L^2(\Gamma_K\backslash G( \RR ))^\infty$.  Borel and Casselman  \cite{BC83} have shown that the reduced $L^2$-cohomology group is isomorphic to the discrete part 
$$H^\ast(\frg,K_\infty ; L^2_{\rm dis} (\Gamma_K\backslash G( \RR ))^\infty)$$
of $H^\ast (Y_K,\CC )$. Note that the  discrete spectrum of $L^2(\Gamma_K\backslash G ( \RR ))$ decomposes as a Hilbert sum of irreducible $G( \RR)$-modules with finite multiplicity, we have 
\begin{equation}\label{eqMa}
\bar{H}^\ast_{(2)} (Y_K,\CC) \cong \bigoplus_{\pi_\RR} m(\pi_\RR) H^\ast(\frg, K_\infty; V^\infty_{\pi_\RR}).
\end{equation}
where $(\pi_\RR,V_{\pi_\RR})$ runs over all the unitary representation of $G( \RR )$ occurring in the discrete spectrum of $L^2(\Gamma_K\backslash G( \RR ))$ with multiplicity $m(\pi_\RR)$. 
The isomorphism \eqref{eqMa} also yields a decomposition of 
$\bar{H}^i(Y_K,\CC)$ and we can denote by $\bar{H}^\ast(\frg, K_\infty ; V_{\pi_\RR})$ the corresponding image of $H^\ast(\frg, K_\infty ; V^\infty_{\pi_\RR})$ in $\bar{H}^i(Y_K,\CC)$. 

\begin{remark}\label{vanishing-shimura}
Combined with Zucker's result and classification of cohomological $(\frg,K_\infty )$-modules, we can obtain a series of vanishing results of cohomology groups on arithmetic manifolds of orthogonal type. For instance, when $G(\RR )=\SO(p,2)$ and $p\geq 3$, we have $\bar{H}^1(Y_K,\CC)\cong H^1(Y_K,\CC)=0$ because there is no $(\frg,K_\infty )$-module with non-zero first relative Lie algebra cohomology (cf.~\cite[$\S$5.10]{BMM11}). This  implies that  the Albanese variety of a connected Shimura variety of orthogonal type is trivial (cf.~\cite{Ko88}).  
\end{remark}

Similar as in \cite{BMM11}, we can define the group  $\bar{H}^k(X_K,\CC)$  as $X_K$ is a finite disjoint union of arithmetic manifolds and define 
$$\bar{H}^i(\mathrm{Sh}(G), \CC)=\lim\limits_{\overrightarrow{~K~}}\bar{H}^i(X_K,\CC) \mbox{ and } \bar{H}^i(\mathrm{Sh}^0(G), \CC)=\lim\limits_{\overrightarrow{~K~}}\bar{H}^i(Y_K,\CC).$$

For a global representation 
$$\pi=\pi_{\infty } \otimes\pi_f \in \cA_2(G),$$ 
we write $\pi^K_f$ as the finite dimensional subspaces of $K$-invariant vectors in $\pi_f$. Then we obtain the decomposition 
\begin{equation}\label{decomposition}
\bar{H}^i(X_K, \CC)\cong \bigoplus m_{\rm dis}(\pi) \bar{H}^i(\frg,K; V_{\pi_{\infty}}^\infty)\otimes \pi_f^K.
\end{equation}
A  global representation $\pi\in\cA_2(G)$ has nonzero  contribution to $\bar{H}^k_{(2)}(X_K,\CC)$ via \eqref{decomposition} and hence to $\bar{H}^k_{(2)}(Y_K,\CC)$  only when it has cohomology at infinity, i.e.
$V_{\pi_{\infty}}^\infty\cong A_\frq$ for some $\theta$-stable parabolic subalgebra $\frq$ in Section 4. Then one can define subspaces of $\bar{H}^i(Y_K,\CC)$ coming from special automorphic representations.

\begin{definition}
1. Let $$\bar{H}^i(\mathrm{Sh}(G),\CC)_{A_\frq}\subseteq \bar{H}^i(\mathrm{Sh}(G),\CC) $$ be the subspace  consisting of cohomology classes contributed from representations $\pi\in\cA_2(G)$, whose infinite component $\pi_{\infty}$ has underlying $(\frg,K_\infty)$-module $A_\frq$. Similarly, we can define  
$\bar{H}^i(\mathrm{Sh}^0(G),\CC)_{A_\frq}$, $\bar{H}^i(X_K,\CC)_{A_\frq}$ and  $\bar{H}^i(Y_K,\CC)_{A_\frq}$.

2. Let $$\bar{H}^i_\theta(\mathrm{Sh}(G),\CC)\subseteq \bar{H}^i(\mathrm{Sh}(G),\CC)$$ be the subspace generated by the image of $\bar{H}^i(\frg, K_\RR; \pi)$ where $\pi$ varies among the irreducible representations in $\cA_2(G)$ which are in the image of $\psi$-cuspidal theta correspondence from a symplectic group.
\end{definition}
Theorem \ref{surjoftheta} implies the following

\begin{theorem}\label{1st-coh-surj}
With notations as above, suppose that the Levi subgroup associated to $\frq$ is $\mathrm{U}(1)^r\times\SO(p-2r,q)$ with $2r < p$ and $3r < m-1$. 
Then the natural map 
\begin{equation}\label{surjectivity}
\bar{H}^i_\theta(\mathrm{Sh}(G), \CC)\cap \bar{H}^i(\mathrm{Sh}(G),\CC)_{A_\frq} \rightarrow \bar{H}^i(\mathrm{Sh}^0(G),\CC)_{A_\frq}
\end{equation}
is surjective.
\end{theorem}
\begin{proof}
The surjectivity of \eqref{surjectivity} is proved in \cite[Theorem 8.10]{BMM11} for the cuspidal part of the reduced $L^2$-cohomology groups, i.e. the subspaces generated by the contribution from cuspidal automorphic representations in $\cA_2(G)$. Since Theorem \ref{surjoftheta} now holds for all square integrable representations, the proof of \cite[Theorem 8.10]{BMM11} extends to the whole reduced $L^2$-cohomology group. We sketch the proof below.

By Theorem \ref{surjoftheta} the space $\bar{H}^i(\mathrm{Sh}(G),\CC)$ is generated by the image of $\bar{H}^i(\frg,K_\infty ; \sigma\otimes \chi)$, where $\sigma$ is in the image of $\psi$-cuspidal theta correspondence from a symplectic group and $\chi$ an automorphic quadratic character. Then for any $\omega\in H^i(\frg,K_\infty ; \sigma\otimes \chi)$, one can construct a twisting cohomology class $\omega \otimes \chi^{-1}$ (cf.~\cite[$\S$8.9]{BMM11}) in $\bar{H}^i(\frg, K_\infty ;\sigma)$ such that $\omega$ and $\omega \otimes \chi^{-1}$ have the same image in $\bar{H}^i (\mathrm{Sh}^0(G),\CC)$. This immediately yields the assertion.
\end{proof}

\section{Special cycles on arithmetic manifolds of orthogonal type}
In this section, we briefly review the theory of  {\it special theta lifting} of Kudla and Millson and its connection to special cycles on arithmetic manifolds of orthogonal type. 

\subsection{Special cycles on arithmetic manifolds}  \label{S8} We keep notations as in $\S$7.1. Given a vector $\bv\in V^r$, we let $U=U(\bv)$ be the $\QQ$-subspace of $V$ spanned by the components of $\bv$. Let $D_{\bv}\subset D$ be the subset consisting of $q$-planes which lie in $U^\perp$. The codimension $rq$ natural cycle $c(U,g_j,K)$ on $\Gamma_{g_j}\backslash D$ is defined to be the image of 
\begin{equation}\label{special-manifold}
\Gamma_{g_j,U}\backslash D_\bv\rightarrow \Gamma_{g_j}\backslash D.
\end{equation}
where $\Gamma_{g_j,U}$ is the stabilizer of $U$ in $\Gamma_{g_j}$.   When $K$ is small enough, \eqref{special-manifold}
is an embedding and hence the natural cycles on $\Gamma_{g_j}\backslash D$ are just arithmetic submanifold of the same type. 

For any $\beta\in \mathrm{Sym}_{r\times r} (\QQ)$, we set 
$$\Omega_\beta=\{\bv\in V^r \; | \;  \frac{1}{2} (\bv,\bv)=\beta,\dim U(\bv)=\rank \beta\}.$$
To any $K$-invariant Schwartz function $\varphi \in \cS(V(\AA_f)^r)$ and any $\beta\in \mathrm{Sym}_{r\times r} (\QQ)$ we associate a {\it special cycle} on $X_K$ defined as the linear combination:
\begin{equation} \label{specialcycle}
Z(\beta, \varphi,K)= \sum\limits_j \sum\limits_{\bv\in \Omega_\beta(\QQ) \atop\mod\Gamma'_{g_j}} \varphi(g_j^{-1}\bv) c(U(\bv),g_j,K).
\end{equation}

\begin{remark} \label{Rsc} 
Let $\Lambda \subset V$ be an even lattice. For each prime $p$, we let $\Lambda_p = \Lambda \otimes \ZZ_p$ and let $K_p$ be the subgroup of $G(\QQ_p )$ which leave $\Lambda_p$ stable and acts trivially on $\Lambda_p^{\vee}$. Then
$$\Gamma_K = \{ \gamma \in \mathrm{SO} (\Lambda ) \; | \; \gamma \mbox{ acts trivially on } \Lambda^{\vee} / \Lambda \}$$
and $Y_K = \Gamma_K \backslash D$. Moreover: an element $\beta$ as above is just a rational $n$ and a $K$-invariant function $\varphi \in \cS(V(\AA_f))$ corresponds to a linear combination of characteristic functions on 
$\Lambda^\vee / \Lambda$. Special cycles in $Y_K$ are therefore linear combinations of the special cycles 
$$\sum_{\substack{x \in \Omega_n \mod \Gamma_K \\ x \equiv \gamma \mod \Lambda}} \Gamma_x \backslash D_x,$$
as $n \in \mathbb{Q}$ and $\gamma \in \Lambda^\vee / \Lambda$ vary. Here we have denoted by $\Gamma_x$ is the stabilizer of the line generated by $x$ in $\Gamma_K$.

In particular, in the moduli space of quasi-polarized K3 surfaces the NL-divisors $D_{h,d}$ of Maulik and Pandharipande (see \eqref{NLdiv}) are particular special cycles and any special cycle is a linear combination of these.
\end{remark}

Let $t$ be rank of $\beta$. Kudla and Millson \cite{KM90} have associated  a Poincar\'e dual cohomology class  $\{ Z[\beta,\varphi,K]\}$  in $H^{tq}(X_K,\CC)$ and we define  
$$[\beta, \varphi]= \{Z[\beta,\varphi,K]\}\wedge e_q^{r-t} \in \bar{H}^{rq}(X_K, \CC).$$ 

\begin{remark}\label{special-class} The class of the Euler form $e_q$ belongs to the subspace spanned by the submanifolds $c(U,g_j,K)$ when $q$ is even (see Corollary \ref{Euler-special} below). Therefore, $[\beta,\varphi]$ can be viewed as the class of a linear combination of arithmetic manifolds of the same type. 
\end{remark}	

\begin{definition}
Let 
$$SC^{rq}(\mathrm{Sh}(G))\subseteq H^{rq}(\mathrm{Sh}(G),\CC)$$ 
be the subspace spanned by the $[\beta,\varphi]$ and set  
\begin{equation}\label{spaceofcycles}
SC^{rq}(X_K):=SC^{rq}(\mathrm{Sh}(G))^K,
\end{equation}
to be the $K$-invariant subspace. Then we define the space of {\it special cycles} on $Y_K$ to be the projection of  $SC^{rq}(X_K)$ to $H^{rq}(Y_K,\CC)$, which is denoted by 
$SC^{rq}(Y_K)$. 
\end{definition}

According to Remark \ref{special-class}, the subspace $SC^{rq}(Y_K)$ is spanned by  arithmetic submanifolds in $Y_K$ of the same type and codimension $rq$.  

\subsection{The class of Kudla-Millson  and special theta lifting} 
Let $\Omega^{rq}(\widehat{D}, \CC)$ be the space of smooth $nq$-forms on $\widehat{D}$.  In \cite{KM90}, 
Kudla and Millson have constructed a Schwartz form:
$$\varphi_{rq}\in  [\cS(V(\RR )^r )\otimes \Omega^{rq}(\widehat{D},\CC)]^{G(\RR )},
$$
(see also \cite[Chapter 9]{BMM11}). 
Fixing a level $K\subseteq G(\AA_f)$ and a $K$-invariant Schwartz function $\varphi\in\cS(V(\AA_f)^r)$, we define a global Schwartz form:
\begin{equation}\label{globalschwartz}
\phi=\varphi_{rq}\otimes  \varphi \in [\cS(V(\AA)^r)\otimes \Omega^{rq}(\widehat{D},\CC)]^{G(\RR )}.
\end{equation}
Using the Weil representation, we can form a theta form
\begin{equation}\label{thetaform}
\begin{aligned}
\theta_{\psi,\phi}(g,g')=\sum\limits_{\xi\in V(\QQ)^r } \omega_\psi(g,g')(\phi)(\xi) 
\end{aligned}
\end{equation}
as in \eqref{thetafunction}. As a function of $g\in G(\RR)$, it defines a closed $rq$-form on $X_K$, denoted by $\theta_{rq}(g',\varphi)$. 
Let $[\theta_{rq}(g',\varphi)]$ be the associated class in $ \bar{H}^{rq}(X_K, \CC)$.  

Note that there is  perfect pairing from Poincar\'e duality
\begin{equation}\label{perfect-pair}
\left<,\right>:H^{rq}(X_K,\CC) \times H^{(p-r)q}_c(X_K,\CC)\rightarrow \CC,
\end{equation}
coming from Poincar\'e duality, where $H^\ast_c(-)$ denotes the de Rham cohomology group with compact support.  We now recall the main result of \cite{KM90} (see also \cite{FM06}):

\begin{proposition}\label{KM-mod}\cite[Theorem 2]{KM90}
As a function of $g'\in \Mp_{2r}(\AA)$, the cohomology class $[\theta_{rq}(g',\varphi)]$ is a holomorphic Siegel modular form of weight $\frac{m}{2}$ for some congruence subgroup  with coefficients in $\bar{H}^{rq}(X_K, \CC)$. 
Moreover, for any  closed $(p-r)q$-form $\eta$  with compact support on $X_K$, the Fourier expansion of $\left<[\theta_{rq}(g',\varphi)],\eta\right>$ is 
\begin{equation}\label{fourier-coefficients}
\left<[\theta_{rq}(g',\varphi)], \eta\right>=\sum\limits_{\beta\geq 0}\left<[\beta,\varphi], \eta\right>W_\beta(g'),
\end{equation}
where 
$$W_\beta(g')=\omega_{\infty}(g'_\infty) \exp (-\pi \mathrm{tr} (x_\infty , x_\infty )) $$ 
and $\omega_\infty$ is the Weil representation on $\Mp_{2r}(\RR)$.
\end{proposition}

\begin{corollary}\label{Euler-special}
When $q$ is even, the Euler form  $e_q$ is a linear combinantion of  the cohomology class of  special cycles on $X_K$ of codimension $q$. 		
\end{corollary}	
\begin{proof} 
	We prove it by contradiction. Suppose $e_q$ is not in the span of  $\{Z(\beta, \varphi, K)\}$. 
	Then there exists an element $\eta\in H^q(X_K,\CC)^\ast$ such that $\eta(e_q)\neq0$ and $\eta$ vanishes on the span of special cycles. One can identify $\eta$ as an element in $H_c^{(p-n)q}(X_K,\CC)$ via the perfect pairing \eqref{perfect-pair}. Then by Proposition \ref{KM-mod},  we know that $\left<[\theta_{q}(g',\varphi)],\eta\right>$ is a holomorphic modular form of weight $\frac{m}{2}$ for some congruence subgroup. 
	Moreover, the constant term of $\left<[\theta_{q}(g',\varphi)],\eta\right>$ is $\eta(e_q)\neq 0$ (up to a nonzero scalar) while  all the other Fourier coefficients are zero by assumption and \eqref{fourier-coefficients}. This is impossible and hence the corollary is proved. 
\end{proof}

\begin{remark}\label{Hodge-bundle} When $q=2$ and hence $D$ is Hermitian, we know that $D$ is a domain in $\PP(V\otimes \CC)$. There is an ample line bundle $\LL$ on $Y_K$ --- the Hodge bundle --- which is the descent of $\cO_{\PP(V\otimes \CC)}(1)$ (cf.~\cite[$\S$4.3]{MP13}).  The Euler form $e_2$ is just the first chern class $c_1(\LL)$ of $\LL$ (up to a scalar), which is the K\"{a}hler class.  Corollary \ref{Euler-special} then implies that $\LL$ is spanned by connected Shimura subvarieties  in $Y_K$ of codimension one associated to $\SO(p-1,2)$.   
\end{remark}

\subsection{Proof of Theorem \ref{main-theorem}} It suffices to show that  the natural  projection  
\begin{equation}
SC^{rq}(Y_K)\rightarrow H^{rq}(Y_K,\CC)_{A_\frq}
\end{equation}
 is surjective. Let us define 
\begin{equation}\label{theta-subspace}
\bar{H}^{rq}_{\theta_{rq}}(\mathrm{Sh}(G), \CC)_{A_\frq}\subseteq \bar{H}^{rq}_\theta (\mathrm{Sh}(G), \CC)_{A_\frq} ,
\end{equation}
 to be the subspace generated by the image of  $[\theta_{rq}(g',\varphi)]$ for $\varphi\in \cS(V(\AA_f)^n)$ where the infinite component  of the global Schwartz form $\varphi$ is given by
$\varphi_{\infty}=\varphi_{rq}$; which is called the space of Kudla-Millson's special theta lifting classes in $\bar{H}^{rq}(\mathrm{Sh}(G), \CC)$. By definition, the subspace $\bar{H}^{rq}_{\theta_{rq}}(\mathrm{Sh}(G), \CC) $ is spanned by the image of the class
$$[\theta_{rq}^f(\varphi)]:=\int_{\Mp(X)\backslash\Mp_{2r}(\AA)} [\theta_{rq}(g',\varphi)]f(g')dg' ,$$
for some  $f\in L^2_{\rm cusp}(\Mp(X)\backslash\Mp_{2r}(\AA))$. Then we claim that:
 \begin{enumerate}[(I)]
 \item\label{claim1}  
 $\label{surjofspecialtheta}
 \bar{H}^{rq}_{\theta_{rq}}(\mathrm{Sh}(G), \CC)_{A_\frq}=\bar{H}_\theta^{rq} (\mathrm{Sh}(G), \CC)_{A_\frq}.
$
 \item \label{claim2} The projection
 \begin{equation}\label{surj-theta}
SC^{rq}(\mathrm{Sh}(G)) \rightarrow \bar{H}^{rq}_{\theta_{rq}} (\mathrm{Sh}(G),\CC)_{A_\frq},
 \end{equation}
is surjective.
 \end{enumerate}	
 
Combining $({\rm I})$, $({\rm II})$ with Theorem \ref{1st-coh-surj},    one can deduce that
\begin{equation}\label{surjofspecialcycle}
SC^{rq}(Y_K)\rightarrow \bar{H}_{\theta_{rq}}^{rq}(Y_K, \CC)_{A_\frq}=\bar{H}^{rq}(Y_K, \CC)_{A_\frq}
\end{equation}
 is surjective as long as $2r<p$ and $3r<m-1$ by taking $K$-invariant classes on both sides of \eqref{surj-theta} and projecting to the cohomology groups of $Y_K$.
\\

\noindent{\it Proof of the claims.} Claims (\ref{claim1}) and (\ref{claim2}) are actually extensions of the results in \cite[\S 10]{BMM11} for subspaces of cuspidal classes to the entire reduced $L^2$-cohomology group.  More precisely, 
the identity in part  (\ref{claim1}) is already proved in \cite[Theorem 10.5]{BMM11} for cuspidal classes, but the proof does not depend on cuspidality at all. So one can directly apply the argument there to obtain (\ref{claim1}) and we omit the details here.

For part (\ref{claim2}), the proof  of the inclusion \eqref{surjoftheta} is analogous to \cite[Theorem 10.7]{BMM11}. 
Using the perfect pairing  $\left<,\right>$ in \eqref{perfect-pair}, one can define
$$SC^{rq}(\mathrm{Sh}(G))^\perp, H^{rq}_{\theta_{rq}}(\mathrm{Sh}(G),\CC)^\perp_{A_\frq},$$ 
to be the annihilator of  the corresponding spaces in  $H^{(p-r)q}_{c}(\mathrm{Sh}(G),\CC)$. 
For any  $\eta\in SC^{rq}(\mathrm{Sh}(G))^\perp$, assume that $\eta$ is $K'$-invariant for some level $K'$. According to Proposition \ref{KM-mod},  we know that the Siegel modular form
\begin{equation}
\begin{aligned}
\theta_\varphi(\eta):&=\left<[\theta_{rq}(g',\varphi)], \eta\right>=\int_{X_{K'}}\theta_{rq}(g',\varphi)\wedge \eta\\ &=\sum\limits_{\beta\geq 0}\left<[\beta,\varphi], \eta\right>W_\beta(g)
\end{aligned}
\end{equation}
is zero by our assumption.
Then we have 
$$\left<[\theta^f_{rq}(\varphi)],\eta\right>=\int_{\Mp(W)\backslash\Mp_{2r}(\AA)} \theta_\varphi(\eta)f(g')dg'=0$$
which implies that $\eta$ is orthogonal to the all  $[\theta_{rq}^f(\varphi)] \in \bar{H}^{rq}_{\theta_{rq}}(\mathrm{Sh}(G), \CC)$. It follows that 
$$SC^{rq}(\mathrm{Sh}(G))^\perp\subseteq \bar{H}^{rq}_\theta(\mathrm{Sh}(G))^\perp_{A_\frq}$$ 
which implies the surjectivity of \eqref{surj-theta}. 
\qed

\subsection{Non-compact hyperbolic manifolds and Shimura varieties} 
For cohomology classes on hyperbolic manifolds and locally Hermitian symmetric varieties, we have the following results.

\begin{theorem}\label{hyperbolic}
Let $Y$ be an arithmetic hyperbolic manifold associated to the group $\SO(p,1)$ as above. Then, for all $r< p/3$, the group $H^{n}(Y,\QQ)$ is spanned by the classes of special cycles. 
\end{theorem}
\begin{proof} We may assume that $p>3$ (otherwise the statement is trivial). Then $r<p/3$ implies $r<[p/2]-1$, so that 
$H^{r}(Y, \CC)\cong \bar{H}^r(Y,\CC)$.  Furthermore,  there is a unique $(\frg,K_\RR)$-module $A_{r}$ for $\SO(p,1)$ such that  $H^r(Y,\CC)_{A_{r} }\neq 0$.  Hence we have
$$H^r(Y,\CC)=\bar{H}^r(Y,\CC)_{A_{r}} =\bar{H}^{r\times 1}(Y,\CC).$$ 
And the assertion follows from Theorem \ref{main-theorem}.
\end{proof}

\begin{theorem}\label{shimura}
Let $Y$ be a smooth Shimura variety associated to the group $\SO(p,2)$ as above. 
Then, for all $r <(m-1)/3$ the subspace $H^{r,r}(Y) \subset H^{2r} (Y , \CC)$ is defined over $\QQ$ and spanned by the classes of  special cycles.  

If we moreover assume $r<\frac{p}{4}$ then $H^{2r}(Y,\QQ) = H^{r,r}(Y)  \cap H^{2r}(Y,\QQ)$ and therefore $H^{2r}(Y,\QQ)$ is spanned by the classes of special cycles.
\end{theorem}
\begin{proof} We may assume that $p\geq 3$ for otherwise the statement is trivial. Then if $r< (p+1)/3$ we have $2r <p-1$ and the cohomology group $H^{2r}(Y,\CC)$ is isomorphic to $\bar{H}^{2r}(Y,\CC)$, see Example \ref{zucker-conj}. Moreover, the Killing form provides a canonical element in $H^{1,1}(\frg,K_\infty ;\CC)$, whose image in $H^2(Y,\CC)$ is just the K\"{a}hler class (see Remark \ref{Hodge-bundle}). The cup product with the K\"{a}hler form induces a Lefschetz structure on $H^{\ast}(Y,\CC)$,
$$L^{k}:H^i(Y,\CC)\rightarrow H^{i+2k}(Y,\CC).$$
For $r<\frac{p-1}{2}$, the pure Hodge structure 
$$H^{2r}(Y,\CC)=\bigoplus_{i+j=2r} H^{i,j}(Y)$$ 
is compatible with the Lefschetz structure. It also coincides with the Hodge structure on the relative Lie algebra cohomology group $H^\ast(\frg,K_\infty ; A_{r,r})$. So $H^{r\times 2}(Y,\CC)=H^{r,r}(Y)$ decomposes as  
$$H^{r,r}(Y)=\bigoplus_{t=0}^r L^{2r-2t} H^{2t}(Y,\CC)_{A_{t,t}},$$
(see \cite[\S 14]{BMM11}).

By Theorem \ref{main-theorem}, we know that the subspace $H^{2t}(Y,\CC)_{A_{t,t}}$ is spanned by the classes of special cycles. Note that the cup product with the K\"{a}hler form is actually to take the intersection with hyperplane class $e_2$, which is a linear combination of special cycles by Corollary \ref{Euler-special}. Since the intersections of special cycles remain in the span of special cycles,  Theorem \ref{shimura} follows. 

When $4r<p$,  the Hodge structure $H^{2r} (Y, \CC)$ is of pure weight $(r,r)$ (but is not (in general) if $4r\geq p$, see \cite[$\S$5.10]{BMM11}). This immediately yields the last assertion. 
\end{proof}

\begin{remark} 
When $(p,q)=(3,2)$, $Y$ is a Siegel modular threefold. In this case, we recover the surjectivity result proved in \cite{HH12}. 
\end{remark}

Now suppose $Y=\Gamma\backslash D$ is a connected Shimura variety but not necessarily smooth. Then $Y$ is a smooth quasi-projective orbifold (cf.~\cite[Section 14]{Hai00}), as we can take  a neat subgroup $\Gamma'\subseteq \Gamma$ so that  $Y'=\Gamma'\backslash D$ is smooth.  In this case, we  get 

\begin{corollary}\label{pic}
The cohomology  group $H^{2r}(Y,\QQ)$ is spanned by Poincar\'e dual of special cycles for $r < \frac{p}{4}$. Moreover, $\Pic_\QQ(Y)\cong H^2(Y,\QQ)$ is spanned by special cycles of codimension one.  In particular, the Noether-Lefschetz conjecture holds on $\cK_g$ for all $g\geq 2$.
\end{corollary}
\begin{proof} With notations as above, by Theorem \ref{shimura},  $H^{2r}(\Gamma'\backslash D,\QQ)=SC^{2r}(Y')$ is spanned by the special cycles of codimension $r$ when $r<\frac{p}{4}$. Note that 
$$H^{2r}(Y,\QQ)=H^{2r}(Y',\QQ)^{\Gamma/ \Gamma'} \quad \mbox{ and }  \quad SC^{2r}(Y)=SC^{2r}(Y')^{\Gamma/\Gamma'},$$  
it follows that 
\begin{equation}\label{surj-orbifold}
SC^{2r}(Y)=H^{2r}(Y,\QQ).
\end{equation}
	
Next, since $Y$ is a smooth quasi-projective orbifold and $H^1(Y,\QQ)=0$ (cf.~ Remark \ref{vanishing-shimura}), the first chern class map 
$$c_1:\Pic_\QQ(Y)\rightarrow H^2(Y,\QQ),$$
is an injection by \cite[Proposition 14.2]{Hai00} and hence has to be an isomorphism by \eqref{surj-orbifold}.
\end{proof}

\subsection{More applications}Combined with Borcherds' theta lifting theory, one may follow  Bruinier's work to give an explicit computation of the Picard number for locally Hermitian symmetric varieties associated to an even lattice.  

Let $M\subset V$ be an even lattice of level $N$ and write $M^\vee$ for the dual lattice. Let $\Gamma_M\subset \SO(M)$ be the subgroup  consisting of elements in $\SO(M)$ acting trivially on the discriminant group $M^\vee/M$. Then the arithmetic manifold $Y_M$ associated to $M$ is defined to be the quotient $\Gamma_M\backslash D$. In this case, Bruinier has shown that there is a natural relation between the space of  vector-valued cusp forms of certain type and  $SC^q(\Gamma_M\backslash D)$. 

For the ease of readers,  let us recall the vector-valued modular forms with respect to $M$.  The metaplectic group $\Mp_2(\ZZ)$ consists of
pairs $\left(A, \phi(\tau)\right)$,  where $$ A=
\left(\begin{array}{cc}a & b \\c & d\end{array}\right)\in \SL_2(\ZZ),
~~\phi(\tau)=\pm\sqrt{c\tau+d}. $$ Borcherds has defined a Weil representation $\rho_M$ of $\Mp_2(\ZZ)$ on the group ring $\CC[M^\vee/M]$ in \cite[$\S$4]{Bo98}. Let $\HH$ be the  complex upper half-plane.  For
any $k\in \frac{1}{2} \ZZ$,  a vector-valued modular form $f(\tau)$
of weight $k$ and type $\rho_M$  is a  holomorphic function on
$\HH$, such that
$$f(A\tau)=\phi(\tau)^{2k}\cdot \rho_M(g)
(f(\tau)),~\textrm{for all} ~g=(A,\phi(\tau))\in \Mp_2(\ZZ).$$
Let $S_{k,M}$ be the space of  $\CC[M^\vee/M]$-valued cusp forms of weight $k$ and type $\rho_M$. Then we have 

\begin{corollary}\label{Picard-number} 
Let $Y_M$ be the locally Hermitian symmetric variety associated to an even lattice $M$ of signature $(p,2)$. Then
 $$\dim_\QQ \Pic_\QQ(Y_M)=1+\dim_\CC S_{m/2,M}$$  if $M=U\oplus U(N)\oplus E$ for some even lattice $E$, where $U(N)$ denotes the rank two lattice $\left(\begin{array}{cc}0 & N \\N & 0\end{array}\right).$ 

\end{corollary}
\begin{proof}
By \cite[Theorem 1.2]{Br14}, there is an isomorphism 
\begin{equation}\label{borcherds-surj}
\:S_{m/2,M}\xrightarrow{\sim} SC^2(Y_M)/<\LL>
\end{equation}
via the Borcherd's theta lifting, where $\LL$ is the line bundle defined in Remark \ref{Hodge-bundle}. Note that $SC^2(Y_M)\cong \Pic_\CC(Y_M)$ by Corollary \ref{pic}. It follows that 
 $\dim_\QQ \Pic_\QQ(Y_M)=1+\dim_\CC S_{m/2,M}$.
\end{proof}	

 The dimension of $S_{k,M}$ has been explicitly computed in \cite{Br02}. In particular, when $M=\Lambda_g$, one can get a  simplified formula \eqref{rank} of $\dim S_{k,\Lambda_g}$ in \cite{LT13} $\S$2.5. As $\Lambda_g$ contains two hyperbolic planes, then Theorem \ref{NL-conj} follows from Theorem \ref{shimura} and  Corollary \ref{Picard-number}.
 
\begin{remark}
One can use the similar idea to compute the $q$-th Betti number of arithmetic manifolds associated to a unimodular  even lattice of signature $ (p,q)$.  Note that there is a similar map of \eqref{borcherds-surj} given in \cite[Corollary 1.2 ]{BF10}.
\end{remark}

\bibliographystyle {plain}
\bibliography{NLC}

\end{document}